\documentclass[a4paper,fleqn]{cas-dc}

\usepackage[utf8]{inputenc}
\usepackage[british]{babel}
\usepackage{csquotes}
\usepackage[backend=biber,style=apa,natbib=true]{biblatex}
\usepackage{mathtools}
\usepackage{booktabs}
\usepackage{placeins}
\usepackage{cleveref}
\graphicspath{{figures/}{./}}
\DeclareGraphicsExtensions{.pdf}

\newtheorem{theorem}{Theorem}[section]
\newtheorem{lemma}[theorem]{Lemma}
\newtheorem{corollary}[theorem]{Corollary}

\newdefinition{definition}[theorem]{Definition}
\newdefinition{assumption}[theorem]{Assumption}
\newdefinition{remark}[theorem]{Remark}
\newproof{proof}{Proof}
\AtEndEnvironment{proof}{\qed}

\crefname{section}{Section}{Sections}
\Crefname{section}{Section}{Sections}
\crefname{subsection}{Section}{Sections}
\Crefname{subsection}{Section}{Sections}
\crefname{subsubsection}{Section}{Sections}
\Crefname{subsubsection}{Section}{Sections}
\crefname{figure}{Figure}{Figures}
\Crefname{figure}{Figure}{Figures}
\crefname{table}{Table}{Tables}
\Crefname{table}{Table}{Tables}
\crefname{equation}{Equation}{Equations}
\Crefname{equation}{Equation}{Equations}
\crefname{theorem}{Theorem}{Theorems}
\Crefname{theorem}{Theorem}{Theorems}
\crefname{lemma}{Lemma}{Lemmas}
\Crefname{lemma}{Lemma}{Lemmas}
\crefname{corollary}{Corollary}{Corollaries}
\Crefname{corollary}{Corollary}{Corollaries}
\crefname{definition}{Definition}{Definitions}
\Crefname{definition}{Definition}{Definitions}
\crefname{remark}{Remark}{Remarks}
\Crefname{remark}{Remark}{Remarks}

\crefname{assumption}{Assumption}{Assumptions}
\Crefname{assumption}{Assumption}{Assumptions}

\newcommand{\R}{\mathbb{R}}
\newcommand{\EE}{\mathbb{E}}

\newcommand{\PP}{\mathbb{P}}
\newcommand{\fbm}{B^{H}}
\newcommand{\grad}{\nabla}
\newcommand{\Ff}{\mathcal{F}}
\newcommand{\cU}{\mathcal{U}}
\newcommand{\cHad}{\mathcal{H}_{\mathrm{ad}}}
\newcommand{\cHdet}{\mathcal{H}_{\mathrm{det}}}
\newcommand{\cUdet}{\mathcal{U}_{\mathrm{det}}}

\newcommand{\ad}{\mathrm{ad}}
\newcommand{\norm}[1]{\left\lVert #1 \right\rVert}
\newcommand{\abs}[1]{\left\lvert #1 \right\rvert}
\DeclareMathOperator*{\argmin}{arg\,min}

\hypersetup{
  hidelinks,
  pdftitle={A samplewise backpropagation method for neural networks driven by fractional Brownian motion},
  pdfauthor={Yuecai Han and Jianming Xu},
  pdfkeywords={stochastic neural networks, fractional Brownian motion, stochastic maximum principle, backward stochastic difference equations, stochastic gradient descent}
}

\begin{document}
\let\WriteBookmarks\relax
\def\floatpagepagefraction{1}
\def\textpagefraction{.001}

\shorttitle{Fractional Stochastic Neural Networks}
\shortauthors{Han and Xu}

\title[mode=title]{A samplewise backpropagation method for neural networks driven by fractional Brownian motion}

\author[1]{Yuecai Han}
\ead{hanyc@jlu.edu.cn}
\credit{Conceptualization, Methodology, Formal analysis, Supervision,
Funding acquisition, Writing -- review and editing}

\author[1]{Jianming Xu}
\cormark[1]
\ead{xujm24@mails.jlu.edu.cn}
\credit{Methodology, Validation, Investigation, Data curation,
Visualization, Writing -- original draft}

\affiliation[1]{organization={School of Mathematics, Jilin University},
  addressline={2699 Qianjin Street},
  city={Changchun},
  postcode={130012},
  state={Jilin},
  country={China}}

\cortext[1]{Corresponding author}

\begin{abstract}%
In this paper, we develop a fractional stochastic neural network
with residual dynamics driven by fractional Brownian motion.
By introducing a discrete stochastic maximum principle for the network, we construct the corresponding adjoint recursion. For deterministic network parameters, we prove mean square convergence of projected samplewise stochastic gradient descent. Numerical experiments include a closed form convergence
test, noisy regression with uncertainty quantification, long memory
time series generation and image classification under structured
perturbations.  The results identify settings in which fractional drivers improve long memory recovery or robustness relative to Brownian and deterministic baselines.
\end{abstract}


\maketitle

\section{Introduction}\label{sec:intro}

Deep residual networks admit a dynamical reading in which the layer
index plays the role of time and the layer map plays the role of a
vector field. This perspective, popularised by the neural ordinary
differential equations (ODEs)~\citep{chen2018neuralode,he2016resnet,marzouk2024distribution}, has been
extended to stochastic dynamics through neural
stochastic differential equations (SDEs)~\citep{liu2019neuralsde,tzen2019neural,li2020scalablesde,kidger2021neuralsde}.
Embedding a Brownian driver in the layer dynamics endows each
prediction with an intrinsic distribution, and has been used as a
principled source of epistemic uncertainty in deep
classifiers~\citep{kong2020,archibald2022,archibald2024,detommaso2024fortuna}. Training such
a model amounts to a stochastic optimal control problem with a gradient
that admits a backward equation representation, generalising reverse mode
automatic differentiation~\citep{pardoux1990adapted,yongzhou1999stochastic,archibald2024}.
A wide range of physical and financial signals nevertheless escape the
Brownian template: log volatility in finance is
rough~\citep{gatheral2018roughvol,comte1998fractional}, geophysical
records exhibit long range dependence~\citep{beran1994longrange}, and
network traffic and turbulent flows carry persistent
increments~\citep{granger1980introduction,hosking1981fractional}. The
standard Gaussian process with tunable memory is the fractional
Brownian motion (fBm), parameterised by a Hurst exponent that ranges
over the open unit
interval: for $H>1/2$ the
increments are positively correlated at every lag, for $H<1/2$ they
are negatively correlated and the path is rough, and for $H=1/2$ one
recovers standard Brownian motion~\citep{mandelbrot1968fbm,biagini2008fbm}. Hayashi and Nakagawa~\citeyearpar{hayashi2022}
showed that replacing the Brownian driver of a generative neural SDE by
an fBm can improve the representation of long memory samples, but did
not analyse discriminative learning in this setting.

The training theory most directly related to our setting is based on
Brownian dynamics. Samplewise adjoint sensitivity for Brownian systems
is rooted in backward stochastic differential equations
\citep{elkaroui1997backward,zhang2004numerical} and underlies training
algorithms for neural SDEs
\citep{li2020scalablesde,kidger2021neuralsde,archibald2024}.
Archibald et al.~\citeyearpar{archibald2024}, in particular, establish a
samplewise stochastic maximum principle at $H=1/2$ and an
inverse-in-iteration rate for projected stochastic gradient descent (SGD). Related stochastic
architectures include the diffusion limit formulation of Tzen and
Raginsky~\citeyearpar{tzen2019neural} and neural controlled differential
equations for path dependent data~\citep{kidger2020neuralcde}. For
non-semimartingale signals, rough path theory provides a deterministic
pathwise calculus~\citep{frizvictoir2010,gubinelli2004controlling}, used
in neural rough differential equations for irregular time
series~\citep{morrill2021neural}. In stochastic control, the stochastic maximum principle (SMP) for control systems driven by fBm of \citet{hanhusong2013}
and the backward stochastic differential equations (BSDEs) driven by fBm of Hu and Peng~\citeyearpar{hu2009backward} provide
the closest continuous time counterparts.

Passing from $H=1/2$ to $H\ne1/2$ does more than change the covariance
of the injected noise. Fractional Brownian motion is neither a
martingale nor a semimartingale, and its increments are not independent, so
the conditional expectation of the next increment is generally nonzero. The
usual Brownian adjoint recursion based on independent increments and
the associated martingale representation therefore cannot be
transferred unchanged. The continuous fractional maximum principle
reflects the same difficulty through an adjoint driven jointly by the
fBm and its underlying Brownian motion and through Malliavin derivative
terms in the stationarity condition~\citep{hanhusong2013}. Ordinary reverse mode
differentiation can still be run on a realised finite computation
graph, but relating that reverse sweep to the adapted optimality system
and proving its unbiasedness requires a separate argument. 

The contribution of this work has a numerical analysis aspect and an
algorithm implementation aspect. On the numerical analysis aspect, we
focus on deriving the adapted optimality system and an error estimate
for projected samplewise backpropagation. In the Brownian setting, the
independence of successive increments permits the adjoint equation to
be represented through a standard one step martingale projection. This
argument is no longer available for fractional Brownian increments,
as their conditional expectation contains information from the entire observed
history. Our strategy is to decompose each fractional increment into
its predictable part and a fresh white noise innovation. This enables
us to derive an adapted backward stochastic difference equation and a
discrete stochastic maximum principle, and to show that the
samplewise reverse sweep gives an unbiased gradient estimator for
deterministic network parameters. On the algorithm implementation aspect, we formulate the fractional stochastic neural networks (FSNNs) directly
as a controlled stochastic difference equation on the network depth
grid for the full Hurst range. The implementable algorithm differentiates
one realised fractional noise computation graph at each iteration and
therefore requires only one forward sweep and one reverse sweep, without
solving the adapted adjoint over the entire state space or evaluating
Malliavin derivatives. The correlated increments are reconstructed from
independent white noise innovations, and their coefficients can be
precomputed once, while fast fractional noise samplers reduce the cost
for long networks. The resulting framework allows the diffusion to
depend on the layer, state and trainable control, and is examined on
regression, long memory time series and image classification problems.

The rest of this paper is organized as follows:
In \Cref{sec:prelim}, we introduce the discrete fractional Brownian increments,
the underlying white noise innovations and the abstract well-posedness
lemma for the discrete adjoint. In \Cref{sec:model}, we define the FSNNs as a
controlled stochastic difference equation and state the training
problem. \Cref{sec:adjoint} derives the adjoint equation, the discrete
stochastic maximum principle, the samplewise gradient and the training
procedure.
\Cref{sec:analysis} proves the convergence theorem and its main
corollaries. We examine the performance of FSNNs through several numerical experiments in \Cref{sec:experiments},
and \cref{sec:conclusion} concludes.

Generic positive constants are denoted by $C$.  The Hurst exponent is $H\in(0,1)$, with
$H<1/2$, $H=1/2$ and $H>1/2$ corresponding respectively to rough
anti-persistent, Brownian and persistent increments.  The network depth
is $N$, the step size is $h$, and $t_n=nh$.  We write $B^H$ for
fractional Brownian motion, $\xi_n^H$ for its normalised increment,
$\eta_n$ for the underlying white noise innovation, and $\Ff_n$ for the
discrete filtration.  The state, control and objective are denoted by
$X_n$, $u_n$ and $J$.  Partial derivatives are written as subscripts,
and $DJ(u)v$ denotes the directional derivative of $J$ at $u$ in the
direction $v$.  Other notation is introduced where it is used.

\section{Preliminaries}\label{sec:prelim}
In this section, we first recall the relationship between Brownian motion
and fractional Brownian motion, then review fractional Brownian increments and
the underlying white noise innovations, together with a well-posedness
lemma for the linear backward stochastic difference equation (BS$\Delta$E)
that will play the role of the adjoint equation in
\cref{sec:adjoint}. Since the fractional stochastic neural network in this paper is by construction a
finite depth discrete object, the analysis is carried out entirely in
discrete time: there is no continuous time stochastic differential
equation and no time discretisation error. We use the discrete BS$\Delta$E framework of
Han and Li~\citeyearpar{hanli2025discrete}.

\subsection{Brownian Motion and Fractional Brownian Motion}
\label{sec:prelim-bm-fbm}

A standard Brownian motion $W=(W_t)_{t\ge0}$ is a centred Gaussian
process with continuous paths, $W_0=0$, and covariance
$\EE[W_sW_t]=\min\{s,t\}$. Its increments over disjoint time intervals
are independent, and $W_t-W_s$ has variance $t-s$ for $0\le s<t$.
Fractional Brownian motion extends this model by allowing dependence
between increments while retaining Gaussianity and stationary increments
\citep{mandelbrot1968fbm,biagini2008fbm}.
Specifically, a standard fractional Brownian motion
$\fbm=(\fbm_t)_{t\ge0}$ with Hurst exponent $H\in(0,1)$ is a centred
Gaussian process with continuous paths, $\fbm_0=0$, and covariance
\begin{equation}\label{eq:fbm-cov-continuous}
\EE[\fbm_s\fbm_t]
=\frac12\bigl(s^{2H}+t^{2H}-|t-s|^{2H}\bigr),
\qquad s,t\ge0.
\end{equation}
Consequently, its increment variance is
\begin{equation}\label{eq:fbm-increment-variance}
\EE\bigl[(\fbm_t-\fbm_s)^2\bigr]=|t-s|^{2H}.
\end{equation}
It is also $H$-self-similar: for every $a>0$, the processes
$(\fbm_{at})_{t\ge0}$ and $(a^H\fbm_t)_{t\ge0}$ have the same
finite-dimensional distributions.

At $H=1/2$, \eqref{eq:fbm-cov-continuous} reduces to
$\EE[B^{1/2}_sB^{1/2}_t]=\min\{s,t\}$, so $B^{1/2}$ is a standard
Brownian motion. For $H>1/2$, disjoint increments are positively
correlated, and the covariances of equally spaced increments are not
summable, yielding long-range dependence. For $H<1/2$, disjoint
increments are negatively correlated (anti-persistent), with absolutely
summable covariances; this regime is not a long-memory setting
\citep{biagini2008fbm}. Thus $H=1/2$ provides the Brownian baseline,
while $H\ne1/2$ introduces dependence across network layers.
In the finite depth model below, \eqref{eq:fbm-increment-variance}
motivates dividing each increment over a step of length $h$ by $h^H$.
This gives unit marginal variance for every $H$, so that changing $H$
changes the correlation structure without changing the marginal noise
scale.

\subsection{Discrete Fractional Brownian Increments}
\label{sec:prelim-fbm}

Let $(\Omega,\Ff,\PP)$ be a complete probability space and fix a horizon
$N\in\mathbb N$ together with a step size $h>0$. We work with a scalar
fractional Brownian motion $\fbm$ of Hurst exponent $H\in(0,1)$ sampled on
the uniform grid $\{t_n=nh\}_{n=0}^{N}$. The (normalised) increments
\begin{equation}\label{eq:xi-def}
\xi^H_n \;:=\; h^{-H}\bigl(\fbm_{t_{n+1}}-\fbm_{t_n}\bigr),\qquad n=0,1,\dots,N-1,
\end{equation}
form a stationary centred Gaussian sequence with covariance
\begin{equation}\label{eq:rho-cov}
\begin{aligned}
\rho(i,j)&:=\EE[\xi^H_i\,\xi^H_j]\\
&=\frac12\Bigl(\abs{i-j+1}^{2H}+\abs{i-j-1}^{2H}\\
&\qquad\quad{}-2\abs{i-j}^{2H}\Bigr),
\end{aligned}
\end{equation}
in particular $\rho(i,i)=1$, $\rho(i,j)=0$ for $i\ne j$ and $H=1/2$,
$\rho(i,i\pm 1)<0$ for $H<1/2$ (negatively correlated) and $\rho(i,i\pm 1)>0$ for
$H>1/2$ (positively correlated). Throughout, $\Sigma_N=[\rho(i,j)]_{0\le i,j<N}\in\R^{N\times N}$
denotes the covariance matrix of $(\xi^H_0,\dots,\xi^H_{N-1})$.

\begin{lemma}[Strict positive definiteness]\label[lemma]{lem:Sigma-spd}
$\Sigma_N$ is strictly positive definite for every $H\in(0,1)$ and every
$N\in\mathbb N$, so its unique lower triangular Cholesky factor
\eqref{eq:cholesky} exists. The largest eigenvalue obeys
$\lambda_{\max}(\Sigma_N)\le C_H$ uniformly in $N$ for $H\in(0,1/2]$
(spectral density bounded) and
$\lambda_{\max}(\Sigma_N)\le C_H\,N^{2H-1}$ for $H\in(1/2,1)$
(spectral density with integrable singularity at the origin). For
$H=1/2$ one has $\Sigma_N=I_N$ and $\beta(n,n)\equiv 1$. The smallest
eigenvalue $\lambda_{\min}(\Sigma_N)$ is strictly positive for every
finite $N$ but, in the rough regime $H<1/2$, may decay polynomially in
$N$.
\end{lemma}

\begin{proof}
Positive definiteness follows from the spectral density of fractional
Gaussian noise being strictly positive on $(-\pi,\pi]$ for every
$H\in(0,1)$~\citep{beran1994longrange,biagini2008fbm}. The
largest eigenvalue bounds are the standard Toeplitz estimates for
fractional Gaussian noise: when $H\le 1/2$ the covariance sequence is
summable and the largest eigenvalue is bounded uniformly in $N$. When
$H>1/2$ the spectral density has the low frequency singularity
$|\omega|^{1-2H}$, producing the $N^{2H-1}$ growth
\citep{biagini2008fbm}. Strict positive definiteness for each
finite $N$ follows from the non-degeneracy of the Gaussian vector
$(\xi^H_0,\ldots,\xi^H_{N-1})$.
\end{proof}

\subsection{Underlying White Noise and the Matrices \texorpdfstring{$\beta,\gamma$}{beta, gamma}}
\label{sec:prelim-beta}

By \cref{lem:Sigma-spd}, $\Sigma_N$ admits a unique \emph{lower triangular}
Cholesky factor: there exists $\beta:\{(n,k):0\le k\le n<N\}\to\R$ with
$\beta(n,n)>0$ and
\begin{equation}\label{eq:cholesky}
\rho(n,m)\;=\;\sum_{k=0}^{n\wedge m}\beta(n,k)\,\beta(m,k),\qquad 0\le n,m<N.
\end{equation}
Define the i.i.d.\ standard normal sequence
\begin{equation}\label{eq:eta-def}
\eta_n \;:=\; \sum_{k=0}^{n}\gamma(n,k)\,\xi^H_k,\qquad n=0,1,\dots,N-1,
\end{equation}
where $\gamma=\beta^{-1}$ is the (lower triangular) inverse Cholesky
factor. Equivalently,
\begin{equation}\label{eq:xi-eta}
\xi^H_n \;=\; \sum_{k=0}^{n}\beta(n,k)\,\eta_k.
\end{equation}
Let $\Ff_n:=\sigma(\eta_0,\dots,\eta_{n-1})$ with $\Ff_0=\{\emptyset,\Omega\}$
denote the discrete filtration generated by the innovations. By
equation~\eqref{eq:xi-eta}, $\Ff_n=\sigma(\xi^H_0,\dots,\xi^H_{n-1})$ as well.
For $H=1/2$, $\Sigma_N=I_N$ and $\beta=\gamma$ are identity matrices, so that $\eta_n=\xi^H_n$, recovering the classical adapted
control setting.

The one step ahead predictor of the fractional increment is
\begin{equation}\label{eq:zetaH-def}
\zeta^H_n \;:=\; \EE\bigl[\xi^H_n\,\big|\,\Ff_n\bigr]
            \;=\; \sum_{k=0}^{n-1}\beta(n,k)\,\eta_k
            \;=\; \xi^H_n-\beta(n,n)\,\eta_n.
\end{equation}
Note that $\zeta^H_n$ is $\Ff_n$-measurable and that
$\beta(n,n)=\bigl(\rho(n,n)-\sum_{k=0}^{n-1}\beta(n,k)^2\bigr)^{1/2}\in(0,1]$
is the one step prediction error standard deviation under \cref{lem:Sigma-spd}, with
$\beta(n,n)\equiv 1$ when $H=1/2$.

\begin{remark}[Discrete long memory coefficient]\label[remark]{rem:VHN}
The variance factor
\begin{equation}\label{eq:VHN-def}
\begin{aligned}
V_H(N)&:=\max_{0\le n<N}\EE\bigl[(\zeta^H_n)^2\bigr]\\
&=\max_{0\le n<N}\sum_{k=0}^{n-1}\beta(n,k)^2\\
&=\max_{0\le n<N}\bigl(1-\beta(n,n)^2\bigr)
\end{aligned}
\end{equation}
controls the dependence of the gradient estimator variance on the Hurst
exponent (cf.~\cref{lem:sgd-var}). Three facts about $V_H(N)$ are used in the sequel: 
\begin{itemize}
	\item [(a)]  $V_H(N)=0$ when $H=1/2$ (independent increments,
	$\beta(n,n)\equiv 1$).
	\item [(b)] $V_H(N)\in[0,1)$ for every finite $N$ and
	$H\in(0,1)$, with $\beta(n,n)^2\in(0,1]$ by \cref{lem:Sigma-spd}.
	\item [(c)] for fixed $N\ge2$, $V_H(N)\to1$ as $H\uparrow1$, when the increments become
	nearly collinear. 
\end{itemize}
 No analogous endpoint identity is needed in the rough
limit $H\downarrow0$. The bound $V_H(N)\le 1$, which holds throughout $H\in(0,1)$ by
equation~\eqref{eq:VHN-def} and $\rho(n,n)=1$, is a one step predictor bound used
locally below~\citep{hanli2025discrete}. By itself it does not imply
depth uniform forward moments under the variance preserving scaling.
\end{remark}

\subsection{Backward Stochastic Difference Equations}
\label{sec:prelim-bsde}

We use only the martingale projection form of a linear BS$\Delta$E. Let
$p_N\in L^2(\Omega,\Ff_N;\R^d)$ be given. For
$p_{n+1}\in L^2(\Omega,\Ff_{n+1};\R^d)$ define
\begin{equation}\label{eq:bsde-proj}
\bar p_n:=\EE[p_{n+1}\mid\Ff_n],\qquad
q_n:=\EE[p_{n+1}\eta_n\mid\Ff_n].
\end{equation}
Then
\begin{equation}\label{eq:bsde-chaos-rep}
p_{n+1}=\bar p_n+q_n\eta_n+r_n^\perp,\qquad
r_n^\perp\perp\bigl(L^2(\Ff_n)\oplus\eta_nL^2(\Ff_n)\bigr).
\end{equation}
Consider the backward recursion
\begin{equation}\label{eq:bsde-abstract}
\begin{aligned}
p_n={}&M_n\bar p_n
\Theta_n\bigl(\zeta_n^H\bar p_n+\beta(n,n)q_n\bigr)\\
&+r_n,\qquad n=N-1,\dots,0,
\end{aligned}
\end{equation}
where $M_n,\Theta_n$ are bounded $\Ff_n$-measurable matrices and
$r_n\in L^2(\Omega,\Ff_n;\R^d)$.

\begin{lemma}[Well-posedness of fixed depth BS$\Delta$E]\label[lemma]{lem:bsde-wellposed}
Assume that $M_n,\Theta_n,r_n$ are bounded and $\Ff_n$-measurable and
that $p_N\in\bigcap_{m\ge2}L^m(\Omega,\Ff_N;\R^d)$. Then
equation~\eqref{eq:bsde-abstract} has a unique adapted solution
$(p_n,\bar p_n,q_n)_{0\le n<N}$ with finite moments of every order.
In particular, for fixed $(N,h,H)$ there is a finite constant
$C_{N,h,H}$, depending also on the coefficient bounds and terminal
moments, such that
\begin{equation}\label{eq:bsde-energy}
\max_{0\le n\le N}\EE\norm{p_n}^2
+h\sum_{n=0}^{N-1}\EE\norm{q_n}^2
\le C_{N,h,H}.
\end{equation}
\end{lemma}

\begin{proof}
The local decomposition \eqref{eq:bsde-chaos-rep} follows from the
orthogonal splitting
\[
L^2(\Omega,\Ff_{n+1})
=L^2(\Omega,\Ff_n)\oplus\eta_nL^2(\Omega,\Ff_n)\oplus\mathcal H_n^\perp,
\]
applied to the elementary white noise innovation $\eta_n$. Hence
$(\bar p_n,q_n)$ is uniquely determined by $p_{n+1}$. To construct the
solution, suppose inductively that
$p_{n+1}\in\bigcap_{m\ge2}L^m$. By Jensen's inequality and
H\"older's inequality, for every $m\ge2$, we have
\[
\begin{aligned}
\EE\norm{\bar p_n}^m&\le\EE\norm{p_{n+1}}^m,\\
\EE\norm{q_n}^m
&\le\EE\bigl[\norm{p_{n+1}}^m|\eta_n|^m\bigr]<\infty.
\end{aligned}
\]
The predictor $\zeta_n^H$ is Gaussian and therefore has moments of every
order. Another application of H\"older's inequality in
equation~\eqref{eq:bsde-abstract} uniquely defines
$p_n\in\bigcap_{m\ge2}L^m$. Starting from $p_N$, backward induction
proves existence, uniqueness and the moment claim. Since the depth is finite,
inequality~\eqref{eq:bsde-energy} follows by collecting the finitely many second
moments.  Sharper projection estimates under additional hypotheses are
developed by Han and Li~\citeyearpar{hanli2025discrete}.
\end{proof}

\section{Fractional Stochastic Neural Network}\label{sec:model}

This section formulates the FSNN as an $N$-step controlled stochastic
difference equation driven by the fractional Brownian motion and introduces the corresponding training problem.

\subsection{Forward Dynamics}\label{sec:model-fwd}

Fix a horizon $N\in\mathbb N$, a step size $h>0$ and an initial state
$x_0\in\R^d$. Let $\cU\subset\R^p$ be a closed convex set containing the
origin and let $\fbm,\xi^H_n,\eta_n,\beta(n,k),\Ff_n$ be as in
\cref{sec:prelim-beta}. The forward propagation of the FSNNs is the controlled stochastic difference equation
\begin{equation}\label{eq:fsnn}
\begin{aligned}
X_{n+1}={}&X_n+h\,b(n,X_n,u_n)\\
&+\sqrt h\,\sigma(n,X_n,u_n)\,\xi^H_n,\\
X_0={}&x_0,\qquad n=0,1,\dots,N-1,
\end{aligned}
\end{equation}
where the coefficients $b:\{0,\dots,N{-}1\}\times\R^d\times\R^p\to\R^d$ and
$\sigma:\{0,\dots,N{-}1\}\times\R^d\times\R^p\to\R^d$ are
deterministic. Notably, the diffusion coefficient $\sigma$ is permitted to
depend jointly on $(x,u)$. Since
$\xi^H_n=h^{-H}(\fbm_{t_{n+1}}-\fbm_{t_n})$, the displayed noise term is
equivalently $h^{1/2-H}\sigma(n,X_n,u_n)\Delta\fbm_n$ in raw fBm
increments. It is written with $\sqrt h$ because $\xi^H_n$ has unit
variance. This scaling is
variance preserving: with the normalisation
$\EE(\xi^H_n)^2=1$ of equation~\eqref{eq:xi-def}, its nominal per step diffusion
scale for a frozen coefficient is
$h\,\sigma(n,x,u)\sigma(n,x,u)^*$, independent of $H\in(0,1)$.  A
vector valued driver with independent componentwise fractional
increments can be handled by the same argument with heavier notation.

\begin{remark}[Variance preserving scaling]\label[remark]{rem:sigma-scaling}
Writing $\Delta\fbm_n=h^H\xi^H_n$, the model \eqref{eq:fsnn} is
equivalent to
\[
X_{n+1}=X_n+h\,b+\sigma_H\,\Delta\fbm_n,
\qquad \sigma_H:=\sigma h^{1/2-H}.
\]
Hence variance preserving scaling is a
reparameterisation of the unscaled FSNN by the deterministic factor
$h^{1/2-H}$. The two forms are pathwise identical given a sample of
$\fbm$ and differ only in the notation of which power of $h$ appears
multiplicatively next to $\sigma$. The choice \eqref{eq:fsnn} is
adopted throughout because it keeps the one step diffusion scale
comparable for every $H\in(0,1)$, with no case split between rough and
persistent regimes. 
\end{remark}

\begin{assumption}[Regularity]\label[assumption]{ass:A1}
The functions $b,\sigma$ and the terminal/running costs $\Phi:\R^d\to\R$
and $\ell:\{0,\dots,N{-}1\}\times\R^d\times\R^p\to\R$ are twice
continuously differentiable in $(x,u)$ and their partial derivatives up
to order two are bounded by a constant $L>0$ and globally Lipschitz in
$(x,u)$, uniformly in the discrete time index $n$.
\end{assumption}

\begin{assumption}\label[assumption]{ass:A2}
There exists $L'>0$ such that, for every $(n,x,u)$,
\begin{equation}\label{eq:A2-growth}
\norm{b(n,x,u)}\;\le\;L'\bigl(1+\norm{x}+\norm{u}\bigr),
\qquad
\norm{\sigma(n,x,u)}\;\le\;L'.
\end{equation}
\end{assumption}

\begin{assumption}[Admissible controls]\label[assumption]{ass:A3}
Let $\cHad$ be the Hilbert space of adapted sequences
$u=(u_n)_{n=0}^{N-1}$ with $u_n\in L^2(\Omega,\Ff_n;\R^p)$, equipped
with the norm
\[
\norm{u}_{\cHad}^2:=h\sum_{n=0}^{N-1}\EE\norm{u_n}^2.
\]
The admissible control set is
$\cU_{\ad}:=\{u\in\cHad:u_n\in\cU\text{ a.s. for every }n\}$.
\end{assumption}

The discrete adapted setting of \cref{ass:A3} is the natural FSNN
counterpart of the piecewise constant control class used in continuous
fractional optimal control. Here it is forced by the architecture rather
than imposed for technical convenience.

For the samplewise parameter training algorithm, define the
deterministic parameter space and its feasible set by
\begin{equation}\label{eq:Udet}
\begin{aligned}
\cHdet&:=(\R^p)^N,\qquad
\norm{u}_{\cHdet}^2:=h\sum_{n=0}^{N-1}\norm{u_n}^2,\\
\cUdet&:=\{u\in\cHdet:u_n\in\cU\text{ for every }n\}
\cong\cU^N.
\end{aligned}
\end{equation}
The adapted feasible set $\cU_{\ad}$ is the natural domain of the
stochastic maximum principle, whereas $\cHdet$ is the ambient parameter
space and $\cUdet$ is the feasible set used by the projected
stochastic gradient algorithm.

\begin{lemma}[Forward well-posedness]\label[lemma]{lem:fwd-wellposed}
Under \cref{ass:A1,ass:A2,ass:A3}, for every $u\in\cHad$ the recursion
\eqref{eq:fsnn} admits a unique adapted solution $(X_n)_{n=0}^N$ with
\begin{equation}\label{eq:fwd-l2}
\max_{0\le n\le N}\EE\norm{X_n}^2
\le C_T\Bigl(1+\norm{x_0}^2+\norm{u}_{\cHad}^2+N V_H(N)\Bigr),
\end{equation}
where $T=Nh$ and $C_T=C(L,L',T)$ is independent of $N$. The explicit
$N V_H(N)$ term is finite for each fixed architecture.
\end{lemma}

\begin{proof}
Existence and uniqueness are immediately obtained from the recursive definition.
Write $b_n=b(n,X_n,u_n)$, $\sigma_n=\sigma(n,X_n,u_n)$ and split
$\xi^H_n=\zeta^H_n+\beta(n,n)\eta_n$ using equation~\eqref{eq:xi-eta}. Since
$X_n,b_n,\sigma_n,\zeta_n^H$ are $\Ff_n$-measurable while
$\eta_n\perp\Ff_n$ and $\EE\eta_n^2=1$, we have that
\[
\begin{aligned}
\EE\!\left[\norm{X_{n+1}}^2\,\middle|\,\Ff_n\right]
={}&\norm{X_n+h b_n+\sqrt h\,\sigma_n\zeta_n^H}^2\\
&+h\,\beta(n,n)^2\norm{\sigma_n}^2 .
\end{aligned}
\]
The drift bound in \cref{ass:A2} gives
\[
\norm{X_n+h b_n}^2
\le (1+Ch)\norm{X_n}^2+Ch(1+\norm{u_n}^2).
\]
Keeping the predictable cross term explicit and applying Young's
inequality,
\[
2\sqrt h\,\langle X_n+h b_n,\sigma_n\zeta_n^H\rangle
\le h\norm{X_n+h b_n}^2+\norm{\sigma_n\zeta_n^H}^2.
\]
The bounded diffusion part of \cref{ass:A2},
$\EE(\zeta_n^H)^2\le V_H(N)$ and $\beta(n,n)\le1$ therefore yield
\begin{equation*}
\begin{aligned}
\EE\norm{X_{n+1}}^2
\le{}&(1+Ch)\,\EE\norm{X_n}^2\\
&+Ch\,\bigl(1+\EE\norm{u_n}^2\bigr)+C V_H(N).
\end{aligned}
\end{equation*}
Iterating this recursion via discrete Gr\"onwall produces
inequality~\eqref{eq:fwd-l2}.
\end{proof}

\begin{remark}[Fixed depth scope]\label[remark]{rem:fwd-fixed-depth}
The term $N V_H(N)$ in inequality~\eqref{eq:fwd-l2} is a coarse finite depth bound,
not a sharp asymptotic rate. The loss of depth uniformity is unavoidable:
if $b\equiv0$ and $\sigma\equiv\sigma_0$ is constant, then
\[
\begin{aligned}
X_N-X_0
&=\sqrt h\,\sigma_0\sum_{n=0}^{N-1}\xi_n^H\\
&=h^{1/2-H}\sigma_0 B_T^H,\\
\EE\norm{X_N-X_0}^2
&=\norm{\sigma_0}^2 h^{1-2H}T^{2H}.
\end{aligned}
\]
Thus, at fixed $T=Nh$, the exact second moment grows as
$\norm{\sigma_0}^2T N^{2H-1}$ when $H>1/2$. This paper treats
$(N,h,H)$ as fixed architecture parameters and studies convergence in
the SGD iteration number $K$.
\end{remark}

\subsection{Cost Functional and Training Problem}\label{sec:model-cost}

For $u\in\cHad$, the training loss combines a terminal loss $\Phi$
(e.g.\ cross-entropy
against an external label $\Gamma$, treated as part of the random sample)
with a discrete running cost,
\begin{equation}\label{eq:cost}
J(u)\;:=\;\EE\Bigl[\Phi(X^u_N) + h\sum_{n=0}^{N-1}\ell(n,X^u_n,u_n)\Bigr],
\end{equation}
where $X^u$ is the solution of equation~\eqref{eq:fsnn} for the control $u$.

\begin{definition}[Training problems]\label[definition]{def:P}
The adapted stochastic control problem is
\[
\inf_{u\in\cU_{\ad}}J(u).
\]
For samplewise training of deterministic network parameters, the
optimisation problem analysed below is the restriction
\begin{equation}\label{eq:optimal-control}
u^* \;=\; \argmin_{u\in\cUdet} J(u).
\end{equation}
At each SGD step, the gradient is approximated from one trajectory of
$(\xi^H_0,\dots,\xi^H_{N-1})$.
\end{definition}

The discrete maximum principle of \cref{thm:smp} below provides the
optimality system used to design and analyse such a samplewise scheme.

\begin{remark}[Choice of optimality principle]\label[remark]{rem:smp-vs-dpp}
Optimal control problems are classically approached via either the
dynamic programming principle or the stochastic maximum principle, and
here we adopt the latter. Dynamic programming relies on a Markov state
and a regular value function solving a finite dimensional
Hamilton--Jacobi equation. For $H\ne1/2$ the increments have long
memory and are neither Markov nor a semimartingale, so by
equation~\eqref{eq:zetaH-def} the law of $X_{n+1}$ depends on the entire
innovation history, turning
dynamic programming into a path dependent, infinite dimensional problem.
The maximum principle instead requires only adaptedness, absorbing the
memory into the adjoint projections $(\bar p_n,q_n)$ and yielding the
single trajectory gradient of \cref{prop:sample-grad}.
\end{remark}

\section{Adjoint Equation and Samplewise Algorithm}\label{sec:adjoint}

This section derives the discrete adjoint equation associated with the
FSNN. In particular, the memory of the fractional
increments enters only through the one step decomposition
\[
  \xi_n^H=\zeta_n^H+\beta_n\eta_n,\qquad
  \beta_n:=\beta(n,n),
\]
where $\zeta_n^H$ is $\Ff_n$-measurable and $\eta_n$ is independent of
$\Ff_n$.

\noindent\emph{Indexing convention.}\label{conv:time-index}
All coefficients with subscript $n$ are evaluated at the running pair
$(n,X_n,u_n)$ and are $\Ff_n$-measurable. The backward step from layer
$n+1$ to layer $n$ therefore uses the same fractional increment
$\xi_n^H$ that appears in the forward step from $X_n$ to $X_{n+1}$.
This convention is important: no quantity indexed by $n+1$ appears in
the noise coefficient of the $n$th adjoint step.

\subsection{Adapted Adjoint BS\texorpdfstring{$\Delta$}{Delta}E}
\label{sec:adjoint-eqn}

Fix $u\in\cU_{\ad}$ and let $X$ be the corresponding forward trajectory.
Set
\begin{equation}\label{eq:lin-coef}
\begin{aligned}
A_n&:=I_d+h\,b_x(n,X_n,u_n),\\
B_n&:=\sqrt h\,\sigma_x(n,X_n,u_n),\\
\Lambda_n&:=\ell_x(n,X_n,u_n).
\end{aligned}
\end{equation}
For an adapted adjoint state $p_{n+1}\in L^2(\Omega,\Ff_{n+1};\R^d)$
define its one step projections
\begin{equation}\label{eq:pbar-q-def}
\bar p_n:=\EE[p_{n+1}\mid\Ff_n],
\qquad
q_n:=\EE[p_{n+1}\eta_n\mid\Ff_n].
\end{equation}
The associated conditional noise projection is
\[
\chi_n:=\EE[p_{n+1}\xi_n^H\mid\Ff_n]
=\zeta_n^H\bar p_n+\beta_nq_n.
\]
Equivalently,
\begin{equation}\label{eq:p-chaos}
p_{n+1}=\bar p_n+q_n\eta_n+r_n^\perp,\qquad
r_n^\perp\perp\bigl(L^2(\Ff_n)\oplus \eta_n L^2(\Ff_n)\bigr).
\end{equation}
The adapted adjoint equation is the BS\texorpdfstring{$\Delta$}{Delta}E
\begin{equation}\label{eq:adj-rec}
\left\{
\begin{aligned}
p_N&=\Phi_x(X_N),\\
p_n&=A_n^*\bar p_n+B_n^*\chi_n+h\,\Lambda_n,\\
&\hspace{7.2em}n=N-1,\dots,0.
\end{aligned}
\right.
\end{equation}

\begin{lemma}[Well-posedness of the adapted adjoint]\label[lemma]{lem:adj-bsde-wellposed}
Under \cref{ass:A1,ass:A2,ass:A3}, equation~\eqref{eq:adj-rec} has a unique
adapted solution $(p_n,\bar p_n,q_n)_{0\le n<N}$ with finite moments of
every order. Moreover, for every fixed finite architecture $(N,h,H)$,
\begin{equation}\label{eq:adj-energy}
\max_{0\le n\le N}\EE\norm{p_n}^2
+h\sum_{n=0}^{N-1}\EE\norm{q_n}^2
\le C_{\rm ad},
\end{equation}
where $C_{\rm ad}$ depends on the architecture and coefficient bounds.
\end{lemma}

\begin{proof}
The bounded derivative hypotheses imply that $A_n,B_n,\Lambda_n$ are
bounded and that $p_N=\Phi_x(X_N)$ is bounded. Hence
\cref{lem:bsde-wellposed} applies with $r_n=h\Lambda_n$ and gives
inequality~\eqref{eq:adj-energy}. This is the fixed depth martingale projection form
of the discrete BS$\Delta$E theory of Han and Li~\citeyearpar{hanli2025discrete}.
\end{proof}

\subsection{Discrete Stochastic Maximum Principle}\label{sec:smp}

The next theorem is the grid level stochastic maximum principle for
equation~\eqref{eq:fsnn}. Its most important feature for the sequel is the
normalisation of the diffusion-control term: because the control norm in
\cref{ass:A3} is $h$-weighted while the forward diffusion perturbation is
$\sqrt h\,\sigma_u v_n\xi_n^H$, the Riesz gradient contains
$h^{-1/2}\sigma_u^*(\cdot)$.

\begin{theorem}[Discrete stochastic maximum principle]\label{thm:smp}
Suppose that \cref{ass:A1,ass:A2,ass:A3} hold. Let $u^*$ be a local
minimiser of $J$ over $\cU_{\ad}$ and let $(p,\bar p,q)$ solve
equation~\eqref{eq:adj-rec} along
$(X^*,u^*)$. Then for every bounded feasible perturbation $v$,
\begin{equation}\label{eq:gradJ-formula}
\begin{aligned}
DJ(u^*)v
&=h\sum_{n=0}^{N-1}\EE\langle \mathcal G_n,v_n\rangle,\\
\mathcal G_n
&:=b_u(n)^*\bar p_n+\ell_u(n)
+h^{-1/2}\sigma_u(n)^*\chi_n,
\end{aligned}
\end{equation}
where all derivatives are evaluated at $(n,X_n^*,u_n^*)$. Consequently,
if $u_n^*$ is an interior point of $\cU$ a.s., then
\begin{equation}\label{eq:smp}
\mathcal G_n=0\quad\text{a.s.}
\end{equation}
For a closed convex control set $\cU$, equation~\eqref{eq:smp} is replaced by the
normal cone inclusion
\[
-\mathcal G_n\in N_{\cU}(u_n^*) .
\]
\end{theorem}

\begin{proof}
Let $v$ be bounded with
$u^\varepsilon:=u^*+\varepsilon v\in\cU_{\ad}$ for all sufficiently
small $\varepsilon>0$, and write
\[
Y_n=\lim_{\varepsilon\downarrow0}
\frac{X_n^{u^\varepsilon}-X_n^{u^*}}{\varepsilon}.
\]
The differentiability assumptions imply convergence in $L^2$.  Indeed,
subtracting the two forward recursions, dividing by $\varepsilon$, and
using the bounded second derivatives in \cref{ass:A1} gives a remainder
with conditional second moment $o(1)$ at each step. The usual
discrete Gr\"onwall argument then gives convergence uniformly in
$n=0,\dots,N$.  The limit satisfies the linearised state equation
\begin{equation}\label{eq:var-fwd}
\begin{aligned}
Y_{n+1}={}&A_nY_n+h\,b_u(n)v_n\\
&+\bigl(B_nY_n+\sqrt h\,\sigma_u(n)v_n\bigr)\xi_n^H,\\
Y_0={}&0.
\end{aligned}
\end{equation}
where $B_n=\sqrt h\,\sigma_x(n)$ as in equation~\eqref{eq:lin-coef}. The first
variation of the cost is
\begin{equation}\label{eq:dJ}
\begin{aligned}
DJ(u^*)v={}&\EE\langle p_N,Y_N\rangle\\
&+h\sum_{n=0}^{N-1}\EE\bigl[
\langle \ell_x(n),Y_n\rangle+\langle\ell_u(n),v_n\rangle\bigr].
\end{aligned}
\end{equation}

Since $Y_0=0$, we have that
\[
\EE\langle p_N,Y_N\rangle
=\sum_{n=0}^{N-1}
\EE\bigl[\langle p_{n+1},Y_{n+1}\rangle
        -\langle p_n,Y_n\rangle\bigr].
\]
For the $Y_n$ terms, applying the
decomposition $\xi_n^H=\zeta_n^H+\beta_n\eta_n$ and
equation~\eqref{eq:pbar-q-def}, we obtain
\[
\EE[p_{n+1}\xi_n^H\mid\Ff_n]
=\chi_n.
\]
Therefore
\begin{align*}
&\EE\langle p_{n+1},A_nY_n+B_nY_n\xi_n^H\rangle\\
&\quad=\EE\langle A_n^*\bar p_n+B_n^*\chi_n,Y_n\rangle\\
&\quad=\EE\langle p_n-h\Lambda_n,Y_n\rangle
\end{align*}
by equation~\eqref{eq:adj-rec}. Hence the state variation terms cancel exactly
against $h\sum_n\EE\langle\ell_x(n),Y_n\rangle$ in equation~\eqref{eq:dJ}.
The remaining control variation terms are
\[
\begin{aligned}
&h\sum_{n=0}^{N-1}\EE\langle
b_u(n)^*\bar p_n+\ell_u(n),v_n\rangle\\
&\quad+\sqrt h\sum_{n=0}^{N-1}\EE\langle
\sigma_u(n)^*\chi_n,v_n\rangle .
\end{aligned}
\]
Factoring out the $h$ in the ambient $\cHad$ inner product gives
equation~\eqref{eq:gradJ-formula}. Applying the variational inequality to
time localised, $\Ff_n$-measurable feasible perturbations gives the
pointwise identity in the interior case and the normal cone inclusion
in the closed convex case.
\end{proof}

\subsection{Samplewise Gradient Estimator}\label{sec:sample-grad}

The adapted adjoint in equation~\eqref{eq:adj-rec} is useful for the stochastic
maximum principle, whereas the implementable gradient is obtained by
differentiating the realised single trajectory loss. Given a trajectory
of the white noise innovations,
define the pathwise reverse sweep
\begin{equation}\label{eq:sample-adj}
\begin{aligned}
\widehat p_N&:=\Phi_x(X_N),\\
\widehat p_n&:=\bigl(A_n^*+\xi_n^H B_n^*\bigr)\widehat p_{n+1}
+h\,\Lambda_n,\\
&\hspace{7.2em}n=N-1,\dots,0 .
\end{aligned}
\end{equation}
The corresponding Riesz gradient density in the $h$-weighted control norm
is
\begin{equation}\label{eq:ghat}
\begin{aligned}
\widehat g_n
:={}&b_u(n,X_n,u_n)^*\widehat p_{n+1}\\
&+\ell_u(n,X_n,u_n)
+h^{-1/2}\sigma_u(n,X_n,u_n)^*
     \widehat p_{n+1}\xi_n^H .
\end{aligned}
\end{equation}

\begin{theorem}[Pathwise gradient unbiasedness]\label{prop:sample-grad}
Under \cref{ass:A1,ass:A2,ass:A3}, for every $u\in\cHad$ and every
bounded perturbation $v\in\cHad$,
\begin{equation}\label{eq:sample-unbiased}
DJ(u)v
=\EE\left[
h\sum_{n=0}^{N-1}\langle \widehat g_n,v_n\rangle
\right].
\end{equation}
For deterministic parameter sequences $u\in\cHdet$, this implies
\begin{equation}\label{eq:sample-unbiased-det}
\EE[\widehat g(u)]=\grad J(u)
\quad\text{in }\cHdet.
\end{equation}
\end{theorem}

\begin{proof}
For a fixed noise realisation define the realised loss
\[
\mathcal L(u;\eta)
:=\Phi(X_N^u)+h\sum_{n=0}^{N-1}\ell(n,X_n^u,u_n).
\]
The pathwise variation $Y$ satisfies equation~\eqref{eq:var-fwd}. Applying
ordinary reverse mode differentiation to this deterministic computation
graph gives
\[
D\mathcal L(u;\eta)v
=h\sum_{n=0}^{N-1}\langle\widehat g_n,v_n\rangle,
\]
with $\widehat p$ defined by equation~\eqref{eq:sample-adj}. The bounded
derivative assumptions justify differentiation under the expectation by
dominated convergence, and therefore
$DJ(u)v=\EE[D\mathcal L(u;\eta)v]$. This is equation~\eqref{eq:sample-unbiased}.
If $u,v\in\cHdet$, then $v$ is deterministic and can be taken outside
the expectation in equation~\eqref{eq:sample-unbiased}.  The Riesz
representation theorem on the finite dimensional Hilbert space
$\cHdet$ gives equation~\eqref{eq:sample-unbiased-det}.
\end{proof}

When the diffusion is control independent, $\sigma_u\equiv0$, the
samplewise gradient reduces to the stable form
\begin{equation}\label{eq:ghat-stateonly}
\widehat g_n=b_u(n,X_n,u_n)^*\widehat p_{n+1}
  +\ell_u(n,X_n,u_n).
\end{equation}
The singular factor $h^{-1/2}$ in equation~\eqref{eq:ghat} is absent exactly in
this case. This is the structural point on which the convergence theorem
of \cref{sec:analysis} rests.

\subsection{Training Procedure}\label{sec:train}

The training algorithm is projected SGD driven by the samplewise
gradient above. Choose $u^0\in\cUdet$ and step sizes
$(\eta_k)_{k\ge0}$. At iteration $k$:
\begin{enumerate}
\item Draw an independent white noise vector and reconstruct the
fractional increments:
\[
\begin{aligned}
\eta^{(k)}&=(\eta_0^{(k)},\dots,\eta_{N-1}^{(k)}),\\
\xi_n^{H,(k)}&=\sum_{j=0}^{n}\beta(n,j)\eta_j^{(k)},
\qquad 0\le n<N.
\end{aligned}
\]
\item Run the forward recursion \eqref{eq:fsnn} with the current control
$u^k$.
\item Run the backward sweep \eqref{eq:sample-adj}.
\item Compute $\widehat g^{\,k}$ from equation~\eqref{eq:ghat}. In the
control independent diffusion case use equation~\eqref{eq:ghat-stateonly}.
\item Update
\begin{equation}\label{eq:sgd-update}
u^{k+1}=\Pi_{\cUdet}\bigl(u^k-\eta_k\widehat g^{\,k}\bigr),
\end{equation}
where $\Pi_{\cUdet}$ denotes the projection in the $\cHdet$ norm of
equation~\eqref{eq:Udet}. This is the usual componentwise Euclidean projection
on $\cU$.
\end{enumerate}

The algorithm uses one forward and one backward sweep per sampled
trajectory. The Cholesky coefficients $\beta(n,j)$ can be precomputed
once. If the full correlated sequence is generated by a Davies--Harte
or Wood--Chan sampler, the fractional noise generation cost is
$O(N\log N)$ rather than $O(N^2)$
\citep{daviesharte1987tests,dieker2003simulation,woodchan1994simulation}.

Since $\Delta\fbm_n=h^H\xi_n^H$, the variance preserving model
\eqref{eq:fsnn} is pathwise equivalent to
\[
\begin{aligned}
X_{n+1}={}&X_n+h\,b(n,X_n,u_n)\\
&+\sigma_H(n,X_n,u_n)\Delta\fbm_n,\\
\sigma_H:={}&h^{1/2-H}\sigma .
\end{aligned}
\]
This is only a reparameterisation of the diffusion amplitude. The
gradient formulas above are written for the normalised variables
$\xi_n^H$ because the $h^{-1/2}$ scaling in equation~\eqref{eq:ghat} is then
visible.

\section{Convergence Analysis}\label{sec:analysis}

This section proves the projected stochastic gradient convergence result on the
deterministic parameter space $\cHdet$ for the samplewise gradient of
\cref{sec:sample-grad}. The theorem is stated
for the full trainable diffusion coefficient
$\sigma=\sigma(n,x,u)$, matching the stochastic network framework in
which both drift and diffusion parameters may be learned. The price of
this generality is explicit: because the forward diffusion perturbation
is $\sqrt h\,\sigma_u v_n\xi_n^H$ while the parameter norm in
equation~\eqref{eq:Udet} is $h$-weighted, the Riesz gradient contains the factor
$h^{-1/2}\sigma_u^*(\widehat p_{n+1}\xi_n^H)$. A single trajectory
therefore carries an additional variance contribution of order $h^{-1}$.
The control independent diffusion case $\sigma_u\equiv0$ is recovered as
a corollary with a sharper bound lacking this explicit $h^{-1}$
contribution.

\subsection{Standing Assumptions}\label{sec:ana-ass}

\begin{assumption}[Strong convexity]\label[assumption]{ass:A4}
The deterministic cost functional $J:\cHdet\to\R$ is Fr\'echet
differentiable and $\mu$-strongly convex on $\cUdet$ in the Hilbert norm of
equation~\eqref{eq:Udet}: for all $u,v\in\cUdet$,
\begin{equation}\label{eq:A4-strongcvx}
\langle \grad J(u)-\grad J(v),u-v\rangle_{\cHdet}
\ge \mu\norm{u-v}_{\cHdet}^2 .
\end{equation}
\end{assumption}

\begin{assumption}[Smoothness and moment closure]\label[assumption]{ass:A5}
There exist constants $L_J,M_2,M_4<\infty$ such that, for all
$u,v\in\cUdet$,
\begin{equation}\label{eq:A5-gradlip}
\norm{\grad J(u)-\grad J(v)}_{\cHdet}
\le L_J\norm{u-v}_{\cHdet},
\end{equation}
and, for every $u\in\cUdet$, the pathwise adjoint
\eqref{eq:sample-adj} satisfies
\begin{equation}\label{eq:A5-mom}
\begin{aligned}
\max_{0\le n\le N}\EE\norm{\widehat p_n^u}^2
&\le M_2(1+V_H(N)),\\
\max_{0\le n\le N}\EE\norm{\widehat p_n^u}^4
&\le M_4(1+V_H(N))^2 .
\end{aligned}
\end{equation}
\end{assumption}

\Cref{ass:A5} is the standard bounded moment closure required
by samplewise adjoint SGD\@. In the present finite depth setting it is
understood as a fixed architecture closure: its constants may depend on
$(N,h,H)$, consistently with \cref{rem:fwd-fixed-depth}. Regimes in which
it can be derived from bounded derivatives and linear BS$\Delta$E
estimates are discussed in Han and Li~(\citeyear{hanli2025discrete}, \S2--\S4).

\begin{remark}[On the moment closure]\label[remark]{rem:moment-closure}
The convergence proof uses \cref{ass:A5} only through the second and
fourth moments in inequality~\eqref{eq:A5-mom}.  In the control independent and
additive diffusion regimes, the pathwise recursion
\eqref{eq:sample-adj} is a finite depth linear backward recursion with
bounded coefficient derivatives and Gaussian multipliers, and the
stated closure follows from Gaussian fourth moment estimates and
backward induction.  For the
fully trainable diffusion case, the same bound is kept as an explicit
moment assumption: it is the finite depth analogue of the $L^4$
integrability hypotheses used in fractional maximum principle
estimates, and it isolates moment control from the SGD argument.
\end{remark}

\begin{assumption}[Control independent diffusion]\label[assumption]{ass:A6}
For the improved corollary only, the diffusion coefficient is not a
trainable control:
\begin{equation}\label{eq:A6-stateonly}
\sigma(n,x,u)=\sigma(n,x),
\qquad \sigma_u\equiv0 .
\end{equation}
\end{assumption}

\subsection{Gradient Estimates}\label{sec:ana-aux}

\begin{corollary}[Riesz form of the samplewise gradient]\label[corollary]{cor:gradient-bias}
Under \cref{ass:A1,ass:A2,ass:A3}, for every deterministic parameter
sequence $u\in\cUdet$, the samplewise gradient $\widehat g(u)$ defined
by equation~\eqref{eq:ghat} is unbiased:
\begin{equation}\label{eq:unbiased-general}
\EE[\widehat g(u)\mid u]=\grad J(u)
\quad\text{in }\cHdet.
\end{equation}
\end{corollary}

\begin{proof}
This is equation~\eqref{eq:sample-unbiased-det}.  At an SGD iteration the newly
drawn trajectory is independent of the current deterministic parameter
vector, so conditioning on $u$ gives the displayed form.
\end{proof}

\begin{lemma}[One sample variance]\label[lemma]{lem:sgd-var}
Under Assumptions~\ref{ass:A1}--\ref{ass:A3} and~\ref{ass:A5},
the samplewise gradient
satisfies
\begin{equation}\label{eq:gradvar}
\begin{aligned}
\EE\!\left[\norm{\widehat g(u)-\grad J(u)}_{\cHdet}^2\mid u\right]
&\le (1+V_H(N))
\left(\kappa_0+\frac{\kappa_1}{h}\right),
\end{aligned}
\end{equation}
where $\kappa_0,\kappa_1$ depend only on the coefficient bounds, $T$, and the
moment constants of \cref{ass:A5}. If \cref{ass:A6} holds, then
$\kappa_1=0$.
\end{lemma}

\begin{proof}
By equation~\eqref{eq:ghat},
\[
\widehat g_n
=b_u(n)^*\widehat p_{n+1}+\ell_u(n)
+h^{-1/2}\sigma_u(n)^*\widehat p_{n+1}\xi_n^H .
\]
The bounded derivative assumption gives
$\norm{b_u},\norm{\ell_u},\norm{\sigma_u}\le L$. The drift and running
cost terms satisfy
\[
\begin{aligned}
\EE\norm{b_u(n)^*\widehat p_{n+1}+\ell_u(n)}^2
&\le C\bigl(1+\EE\norm{\widehat p_{n+1}}^2\bigr)\\
&\le C(1+V_H(N)).
\end{aligned}
\]
For the diffusion-control term, Gaussian fourth moments and
\cref{ass:A5} yield
\[
\begin{aligned}
\EE\norm{h^{-1/2}\sigma_u(n)^*\widehat p_{n+1}\xi_n^H}^2
&\le h^{-1}L^2
\bigl(\EE\norm{\widehat p_{n+1}}^4\bigr)^{1/2}\\
&\quad\times\bigl(\EE|\xi_n^H|^4\bigr)^{1/2}\\
&\le C h^{-1}(1+V_H(N)).
\end{aligned}
\]
Multiplying by $h$ and summing over $n=0,\dots,N-1$ gives
\[
\EE\norm{\widehat g(u)}_{\cHdet}^2
\le (1+V_H(N))\left(\kappa_0+\frac{\kappa_1}{h}\right).
\]
The variance bound follows from $\operatorname{Var}(Z)\le\EE\norm{Z}^2$.
If $\sigma_u\equiv0$, the diffusion-control term vanishes and therefore
$\kappa_1=0$.
\end{proof}

\subsection{Main Theorem}\label{sec:ana-main}

\begin{lemma}[Robbins--Monro sequence estimate]\label[lemma]{lem:rm-sequence}
Let $(a_k)_{k\ge k_0}$ be a nonnegative sequence satisfying, for some
$c>1$ and $D\ge0$, with $k_0+1\ge c$,
\begin{equation}\label{eq:rm-seq}
a_{k+1}\le \left(1-\frac{c}{k+1}\right)a_k+\frac{D}{(k+1)^2},
\qquad k\ge k_0 .
\end{equation}
Then there exists a constant $C_{\rm RM}$, depending only on
$a_{k_0},c,D,k_0$, such that
\begin{equation}\label{eq:rm-bound}
a_K\le \frac{C_{\rm RM}}{K+1},\qquad K\ge k_0 .
\end{equation}
\end{lemma}

\begin{proof}
Choose
\[
  A\ge \max\left\{(k_0+1)a_{k_0},\,\frac{D}{c-1}\right\}.
\]
We prove by induction that $a_k\le A/(k+1)$ for all $k\ge k_0$.  If the
claim holds at $k$, then inequality~\eqref{eq:rm-seq} gives
\[
a_{k+1}
\le \left(1-\frac{c}{k+1}\right)\frac{A}{k+1}
   +\frac{D}{(k+1)^2}
=\frac{A}{k+1}-\frac{cA-D}{(k+1)^2}.
\]
Since $A(c-1)\ge D$, we have $cA-D\ge A$, and therefore
\[
a_{k+1}
\le \frac{A}{k+1}-\frac{A}{(k+1)^2}
=\frac{Ak}{(k+1)^2}
\le \frac{A}{k+2}.
\]
The induction is complete, and inequality~\eqref{eq:rm-bound} follows with
$C_{\rm RM}=A$ after enlarging $A$ if necessary to cover the finite
initial segment.
\end{proof}

\begin{theorem}[Projected SGD convergence]\label{thm:main}
Suppose that Assumptions~\ref{ass:A1}--\ref{ass:A3},
\ref{ass:A4} and~\ref{ass:A5} hold. Let $u^*$ be the unique
minimiser of problem~\eqref{eq:optimal-control}, and let $(u^k)_{k\ge0}$ be
generated by the projected stochastic gradient scheme \eqref{eq:sgd-update} with
\begin{equation}\label{eq:eta-choice}
\eta_k=\frac{\eta_0}{k+1},\qquad \eta_0>\frac1\mu .
\end{equation}
Then, for fixed $(N,h,H)$, there exist constants $C_0,C_1>0$, depending
on the constants in the assumptions, $\eta_0$ and the initial error,
such that
\begin{equation}\label{eq:main-rate}
\EE\norm{u^K-u^*}_{\cHdet}^2
\le \frac{1+V_H(N)}{K}
\left(C_0+\frac{C_1}{h}\right).
\end{equation}
\end{theorem}

\begin{proof}
Set $e_k:=u^k-u^*$ and
$a_k:=\EE\norm{e_k}_{\cHdet}^2$. Since the projection onto the closed
convex set $\cUdet$ is non-expansive,
\begin{equation}\label{eq:proj-expansion}
\begin{aligned}
\norm{e_{k+1}}_{\cHdet}^2
&\le\norm{e_k-\eta_k\widehat g^{\,k}}_{\cHdet}^2\\
&=\norm{e_k}_{\cHdet}^2
-2\eta_k\langle \widehat g^{\,k},e_k\rangle_{\cHdet}\\
&\quad+\eta_k^2\norm{\widehat g^{\,k}}_{\cHdet}^2 .
\end{aligned}
\end{equation}
Taking conditional expectation with respect to $u^k$ and using
\cref{cor:gradient-bias},
\[
\EE[\widehat g^{\,k}\mid u^k]=\grad J(u^k).
\]
The first order optimality condition for $u^*$ over $\cUdet$ and
\cref{ass:A4} imply
\begin{equation}\label{eq:strong-cvx-use}
\langle\grad J(u^k),e_k\rangle_{\cHdet}
\ge \mu\norm{e_k}_{\cHdet}^2 .
\end{equation}
Moreover, by inequality~\eqref{eq:A5-gradlip},
\[
\norm{\grad J(u^k)}_{\cHdet}^2
\le 2L_J^2\norm{e_k}_{\cHdet}^2
   +2\norm{\grad J(u^*)}_{\cHdet}^2 .
\]
The last term is finite and can be absorbed into the variance constant
because $1+V_H(N)\ge1$. Combining this estimate with
\cref{lem:sgd-var}, define
\[
D_h:=(1+V_H(N))
\left(\kappa_0'+\frac{\kappa_1'}{h}\right).
\]
Then
\begin{equation}\label{eq:second-moment-sgd}
\begin{aligned}
\EE[\norm{\widehat g^{\,k}}_{\cHdet}^2\mid u^k]
&\le4L_J^2\norm{e_k}_{\cHdet}^2+D_h.
\end{aligned}
\end{equation}
Substitution into inequality~\eqref{eq:proj-expansion} gives
\[
\begin{aligned}
a_{k+1}
&\le \bigl(1-2\mu\eta_k+4L_J^2\eta_k^2\bigr)a_k\\
&\quad+\eta_k^2D_h.
\end{aligned}
\]
For all sufficiently large $k$, both
$4L_J^2\eta_k^2\le\mu\eta_k$ and $k+1\ge\mu\eta_0$ hold, and hence
\[
\begin{aligned}
a_{k+1}
&\le\left(1-\frac{\mu\eta_0}{k+1}\right)a_k\\
&\quad+\frac{\eta_0^2D_h}{(k+1)^2}.
\end{aligned}
\]
Applying \cref{lem:rm-sequence} with $c=\mu\eta_0>1$ and
$D=\eta_0^2D_h$ gives
\[
a_K
\le \frac{1+V_H(N)}{K}
\left(C_0+\frac{C_1}{h}\right),
\]
after absorbing the finite initial segment into the constants.  This is
inequality~\eqref{eq:main-rate}.
\end{proof}

\begin{corollary}[Control independent diffusion rate]\label[corollary]{cor:scaling}
Assume the hypotheses of Theorem~\ref{thm:main} and
Assumption~\ref{ass:A6}. Then the $h^{-1}$ variance term vanishes and
\begin{equation}\label{eq:stateonly-rate}
\EE\norm{u^K-u^*}_{\cHdet}^2
\le \frac{C(1+V_H(N))}{K}.
\end{equation}
Consequently $K\asymp\varepsilon^{-1}(1+V_H(N))$ gives discrete
optimisation error $O(\varepsilon)$.
\end{corollary}

\begin{proof}
Under $\sigma_u\equiv0$, \cref{lem:sgd-var} has $\kappa_1=0$. Repeating the
proof of Theorem~\ref{thm:main} with this sharper variance bound gives
inequality~\eqref{eq:stateonly-rate}. The scaling statement follows by direct
substitution.
\end{proof}

\begin{corollary}[Mini batch rate]\label[corollary]{cor:minibatch-rate}
In the full trainable diffusion setting of Theorem~\ref{thm:main}, suppose each
stochastic gradient step uses the average of $M$ independent trajectory
gradients. Then there exist constants $\widetilde C_0,\widetilde C_1>0$
such that
\begin{equation}\label{eq:minibatch-rate}
\EE\norm{u^K-u^*}_{\cHdet}^2
\le \frac{1+V_H(N)}{K}
\left(\widetilde C_0+\frac{\widetilde C_1}{Mh}\right).
\end{equation}
Thus $M\asymp h^{-1}$ removes the explicit $h^{-1}$ factor from the
leading iteration constant.
\end{corollary}

\begin{proof}
Averaging $M$ independent unbiased gradients divides the conditional
variance in inequality~\eqref{eq:gradvar} by $M$ and leaves the drift contraction
unchanged. The projected stochastic gradient proof is otherwise identical
to that of Theorem~\ref{thm:main}.
\end{proof}

\begin{remark}[Interpretation by Hurst regime]\label[remark]{rem:regime-rates}
When $H=1/2$, the Cholesky factor is the identity,
$\zeta_n^H\equiv0$ and $V_H(N)=0$. The theorem then gives the Brownian
samplewise SGD bound in the same normalised, $h$-weighted control norm.
For $H\ne1/2$, the fractional memory enters through $V_H(N)$ and through
the adjoint moment constants. The factor $V_H(N)$ records the
one step predictability of $\xi_n^H$ from its past and is bounded by one
for every finite depth.
\end{remark}

\begin{remark}[Mechanisms for the sharper rate]\label[remark]{rem:sigma-xu}
The full architecture and the maximum principle are formulated for
$\sigma(n,x,u)$. One explicit contribution preventing an
$h$-uniform single trajectory iteration constant is the variance of
$h^{-1/2}\sigma_u^*(\widehat p_{n+1}\xi_n^H)$. There are two direct ways
to remove this obstruction.  One is the mini batch scaling of
\cref{cor:minibatch-rate}.  The other is an explicit architectural
scaling, for example
\[
  \sigma(n,x,u)=\sigma_0(n,x)+\sqrt h\,\widetilde\sigma(n,x,u),
\]
with bounded $\widetilde\sigma_u$, so that the effective derivative
$\sigma_u=\sqrt h\,\widetilde\sigma_u$ cancels the Riesz factor
$h^{-1/2}$.  This is a model-design restriction, not a consequence of
the general trainable diffusion theorem.  The control independent case
\cref{ass:A6} is the limiting special case in which the
diffusion-control contribution vanishes identically.
\end{remark}

\section{Experiments}\label{sec:experiments}

We use three sets of numerical tests.  E1 contains three parts: a closed form linear quadratic problem in which the
control error in Theorem~\ref{thm:main} and Corollary~\ref{cor:scaling}
can be measured directly, a
high dimensional noisy regression problem, and an uncertainty
quantification test on one dimensional sections.  E2 isolates the
long memory modelling effect by comparing FSNNs with a Brownian SDE-Net
and an RNN on synthetic and real time series.  E3 studies classification
robustness under common image corruptions and additive fBm random field
perturbations.

All fBm increments are generated from the discrete covariance of
\cref{sec:prelim-fbm}.  For long trajectories we use the Davies--Harte
FFT sampler \citep{daviesharte1987tests}. In the scalar LQ test we use
Cholesky sampling because the covariance matrix is small.  Unless
stated otherwise, the diffusion amplitude is written in the
variance preserving form
\[
  \sqrt h\,\sigma(n,X_n,u_n)\xi_n^H,
\]
equivalently $h^{1/2-H}\sigma(n,X_n,u_n)\Delta B_n^H$.  This
normalisation keeps the one step marginal diffusion variance comparable
across Hurst exponents.  The Brownian baseline is $H=1/2$, the regime $H>1/2$
corresponds to persistent long memory increments, and $H<1/2$ corresponds
to anti-persistent rough increments. 

\subsection{E1: Convergence, Regression, and Uncertainty Quantification}
\label{sec:exp-e1}

The three E1 components separate the projected SGD convergence diagnostic,
high dimensional regression, and sectionwise uncertainty quantification under
the same variance preserving fractional noise scaling.

\subsubsection{E1(a): Closed Form LQ Convergence Test}\label{sec:exp-e1p}
We first use a scalar LQ problem with closed form optimal control as an
empirical diagnostic for the quantity controlled by the convergence
theorem.  The forward
equation is
\[
  X_{n+1}=X_n+b\,u_nh+\sigma\sqrt h\,\xi_n^H,\qquad X_0=0,
\]
and the objective is
\[
  J(u)=rh\sum_{n=0}^{N-1}u_n^2
  +q\,\mathbb E\bigl[(X_N-x_T)^2\bigr].
\]
We use $T=1$, $N=8$, and
$(b,r,q,x_T,\sigma)=(1,0.05,1,1,0.5)$.  Since the diffusion is
additive, $\sigma_u\equiv0$, which is precisely the
control independent setting of \cref{cor:scaling}.  The optimum is
constant across layers:
\[
  u_n^*=\frac{S^*}{N},\qquad
  S^*=\frac{qbx_TN}{r+qb^2hN}.
\]
We run unconstrained SGD (equivalently, projected SGD with the identity
projection) with decreasing stepsize
$\alpha_k=c_0/(k+k_0)$, $k_0=50$, using three seeds and two hundred
sample trajectories per seed.  The errors are recorded at
$K\in\{1000,2000,4000,8000,16000\}$.
\Cref{tab:exp-e1p-rate} reports the resulting mean square errors and
estimated convergence slopes.

\begin{table*}[pos=tb]
  \caption{E1(a): direct measurement of
  $\mathbb E\norm{u^K-u^*}^{2}$ in the closed form LQ problem.  The
  last column reports the least squares slope of
  $\log\mathbb E\norm{u^K-u^*}^{2}$ against $\log K$.}
  \label{tab:exp-e1p-rate}
  \centering
  {\small
\begin{tabular*}{\textwidth}{@{\extracolsep{\fill}}lcccccc@{}}
\toprule
 & \multicolumn{5}{c}{$\mathbb{E}\,\|u^{K}-u^{*}\|^{2}$ (mean over 3 seeds)} & slope \\
\cmidrule(lr){2-6}
$H$ & $K{=}1000$ & $K{=}2000$ & $K{=}4000$ & $K{=}8000$ & $K{=}16000$ & $\,d\log\!\cdot\!/\,d\log K$ \\
\midrule
0.3\,(rough) & 1.745e-02 & 4.861e-03 & 1.383e-03 & 3.933e-04 & 1.252e-04 & $-1.79$ \\
0.5\,(Brownian) & 1.854e-02 & 5.463e-03 & 1.685e-03 & 5.280e-04 & 1.907e-04 & $-1.66$ \\
0.7\,(long memory) & 2.100e-02 & 6.822e-03 & 2.366e-03 & 8.426e-04 & 3.390e-04 & $-1.49$ \\
\bottomrule
\end{tabular*}\par}
\end{table*}
\Cref{tab:exp-e1p-rate} checks the iteration dependence in the control
norm.  For all three Hurst regimes, the measured mean square control
error decreases monotonically over the same iteration scale used in the
noisy regression experiment. At the fixed depth $N=8$, the three rows have
comparable scale, so the experiment tests the inverse dependence on $K$
rather than the finer dependence of the constant on $V_H(N)$.

\subsubsection{E1(b): Noisy Function Regression}
In this experiment, we consider an eight dimensional noisy function example:
\begin{align}
f_0(x)
={}&e^{x_1}\cos(2\pi x_2)
  +8x_3\left(x_4-\frac12\right)^2 \notag\\
&+x_5+\log(2+x_6)+x_7^2+2x_8, \notag\\
&\hspace{5em}x\in[0,1]^8 .        \label{eq:e1-target}
\end{align}
The observed response is
$Y=f_0(X)+0.055\,\varepsilon$, where
$\varepsilon\sim\mathcal N(0,1)$ is independent of $X$.  We use a
Latin hypercube design and train FSNNs with
depths $N\in\{4,8,16\}$ and Hurst exponents
$H\in\{0.3,0.5,0.7\}$ for $K\in\{4000,8000,16000\}$ iterations.  Each
configuration is repeated over two seeds.  The training set has
$8192$ Latin hypercube samples, and the test set has $4096$ samples.

This experiment has two roles.  First, it checks stable end to end
training in a nontrivial high dimensional regression problem.  Second,
its trainable state dependent diffusion MLP provides an
architecture level test of the general coefficient $\sigma(n,x,u)$.
The implementation uses direct backpropagation with Adam and a fixed
training driver bank, so this is not a projected SGD convergence test.
The metric in \cref{tab:exp-e1-conv} is held out test MSE against the
noiseless target function.

\begin{table*}[pos=tb]
  \caption{E1(b): held out test MSE for the eight dimensional noisy
  regression problem.  The slope column is the least squares slope of
  $\log(\mathrm{MSE})$ against $\log K$ at fixed $(H,N)$.  Test MSE is
  computed against the noiseless target, and the label noise variance
  $0.055^2=3.0\times10^{-3}$ is a reference scale, not an irreducible
  floor for this metric.}
  \label{tab:exp-e1-conv}
  \centering
  \small
  \begin{tabular}{cccccc}
    \toprule
    $H$ & $N$ & $K=4000$ & $K=8000$ & $K=16000$ & slope \\
    \midrule
    0.3 & 4  & 0.0871 & 0.0114 & 0.0028 & -2.49 \\
    0.3 & 8  & 0.1106 & 0.0105 & 0.0016 & -3.03 \\
    0.3 & 16 & 0.1339 & 0.0179 & 0.0019 & -3.05 \\
    0.5 & 4  & 0.1041 & 0.0121 & 0.0027 & -2.63 \\
    0.5 & 8  & 0.1364 & 0.0143 & 0.0024 & -2.91 \\
    0.5 & 16 & 0.1690 & 0.0222 & 0.0021 & -3.18 \\
    0.7 & 4  & 0.1054 & 0.0132 & 0.0027 & -2.64 \\
    0.7 & 8  & 0.1710 & 0.0143 & 0.0027 & -3.01 \\
    0.7 & 16 & 0.1994 & 0.0234 & 0.0029 & -3.05 \\
    \bottomrule
  \end{tabular}
\end{table*}

\Cref{tab:exp-e1-conv} shows a two stage pattern in these runs.  Between
$K=4000$ and $K=8000$ the error decreases rapidly.  In this
pre asymptotic region larger depth has
larger test error: for example, at $H=0.7$ the MSE at $K=8000$ rises
from $1.32\times10^{-2}$ at $N=4$ to $2.34\times10^{-2}$ at $N=16$.
This is qualitatively compatible with the depth dependent variance
constants in Theorem~\ref{thm:main}.  At $K=16000$ all configurations are of the same
order as the label noise reference scale, so the depth dependence is
less visible in the test metric.

\subsubsection{E1(c): Uncertainty Quantification}
The uncertainty quantification test uses the same target function
$f_0$ in equation~\eqref{eq:e1-target}.  At $K=16000$ and $N=16$ we evaluate the
trained FSNN on eight one dimensional sections of $[0,1]^8$: one coordinate varies in
$[0,1]$ and the other seven are fixed at $0.3$.  The predictive band is
$\hat\mu(x)\pm2\hat\sigma(x)$, where $\hat\sigma$ is the standard
deviation induced by the trained diffusion coefficient.  The numerical
diagnostics in \cref{tab:exp-e1-uq} and the section plots in
\cref{fig:exp-e1-uq} show
calibrated or slightly conservative bands throughout the Hurst grid.
The rough case $H=0.3$ gives the widest bands and the highest coverage,
while the persistent case $H=0.7$ gives the smallest mean error and the
narrowest average band, with coverage still close to the nominal
$95\%$ level.

\begin{table*}[pos=tb]
  \caption{E1(c): sectionwise uncertainty diagnostics at $K=16000$,
  $N=16$.  Coverage is computed against the noiseless target function
  over the eight one dimensional sections.}
  \label{tab:exp-e1-uq}
  \centering
  \small

\begin{tabular}{cccccc}
\toprule
$H$ & MAE $\downarrow$ & RMSE $\downarrow$ &
$2\hat\sigma$ coverage & avg.\ 95\% width & std.\ range \\
\midrule
0.3 & 0.02081 & 0.02439 & 0.9960 & 0.1869 & 0.0314--0.0624 \\
0.5 & 0.03381 & 0.03709 & 0.9879 & 0.1441 & 0.0233--0.0515 \\
0.7 & \textbf{0.01354} & \textbf{0.02270} & 0.9637 & \textbf{0.1423} & 0.0190--0.0537 \\
\bottomrule
\end{tabular}
\end{table*}

\begin{figure*}[pos=p]
  \centering
  \includegraphics[width=0.82\linewidth]{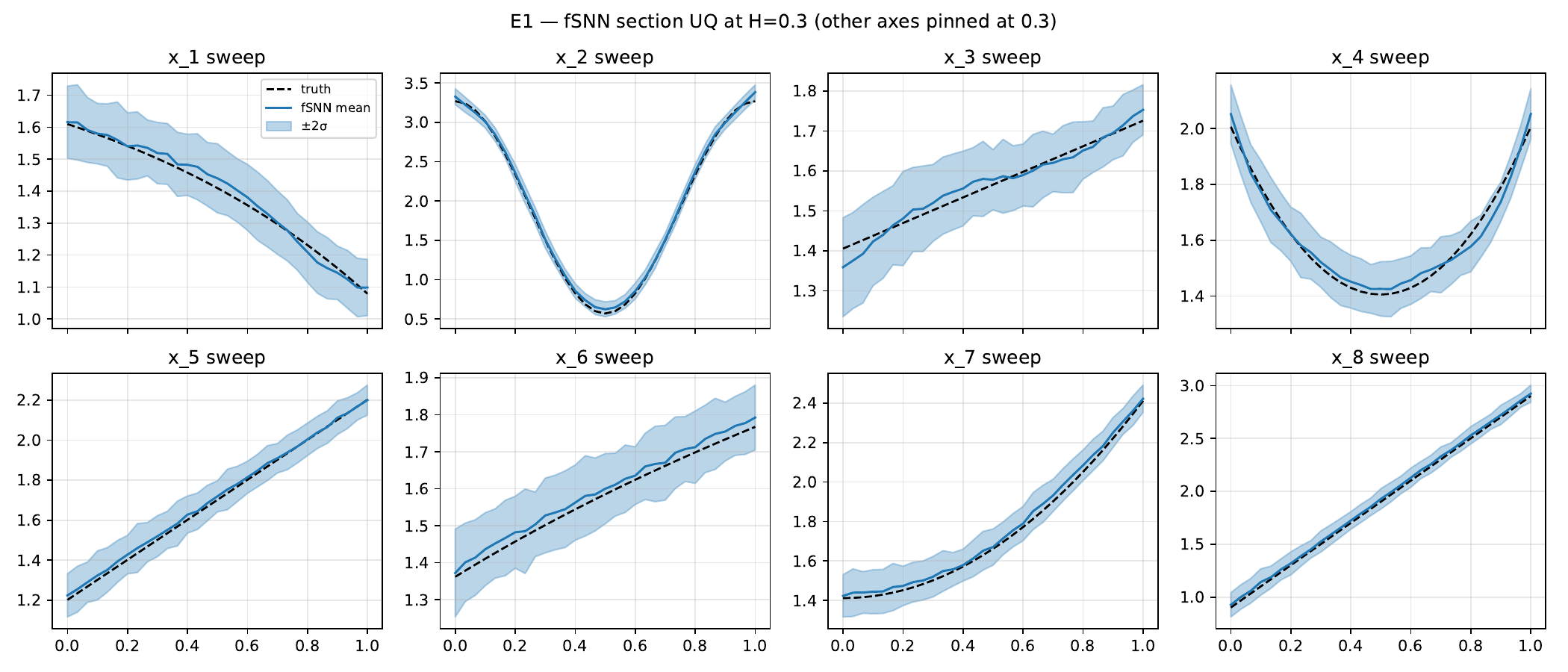}
  \par\smallskip
  \includegraphics[width=0.82\linewidth]{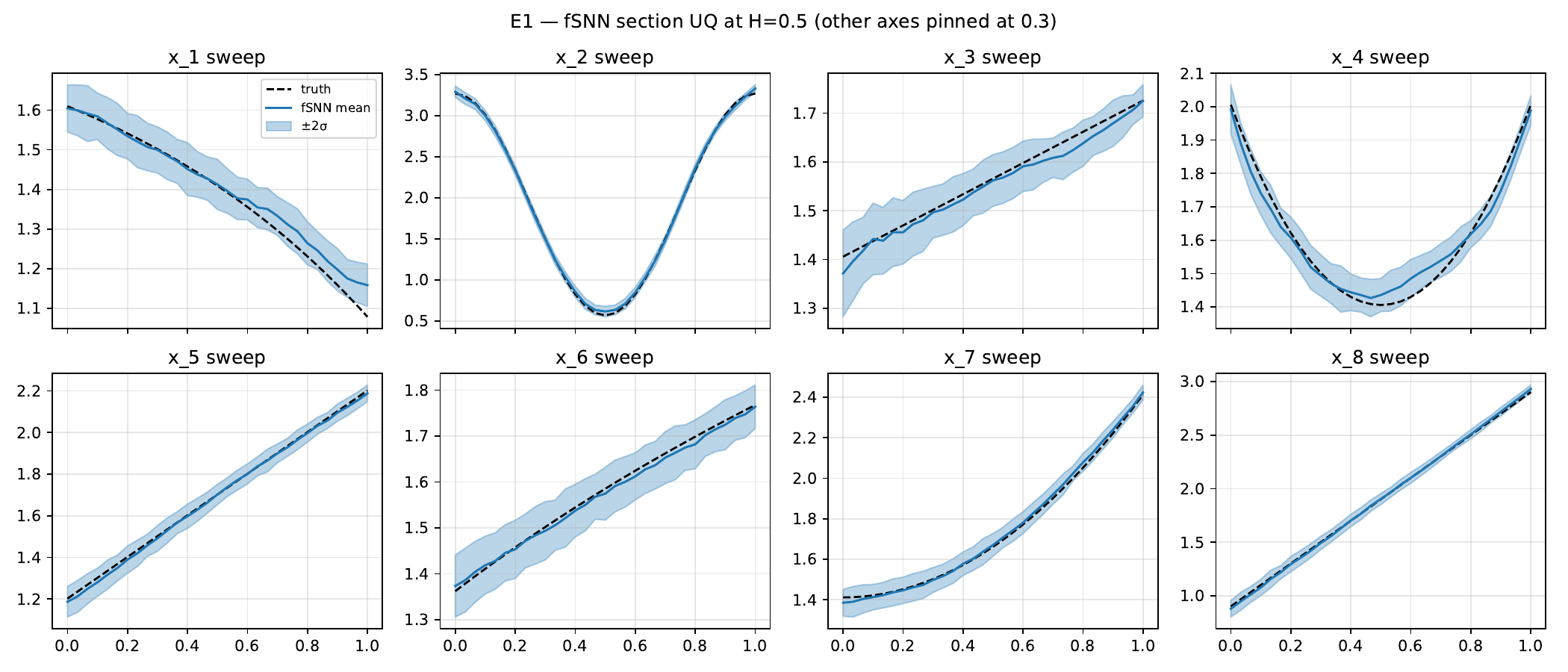}
  \par\smallskip
  \includegraphics[width=0.82\linewidth]{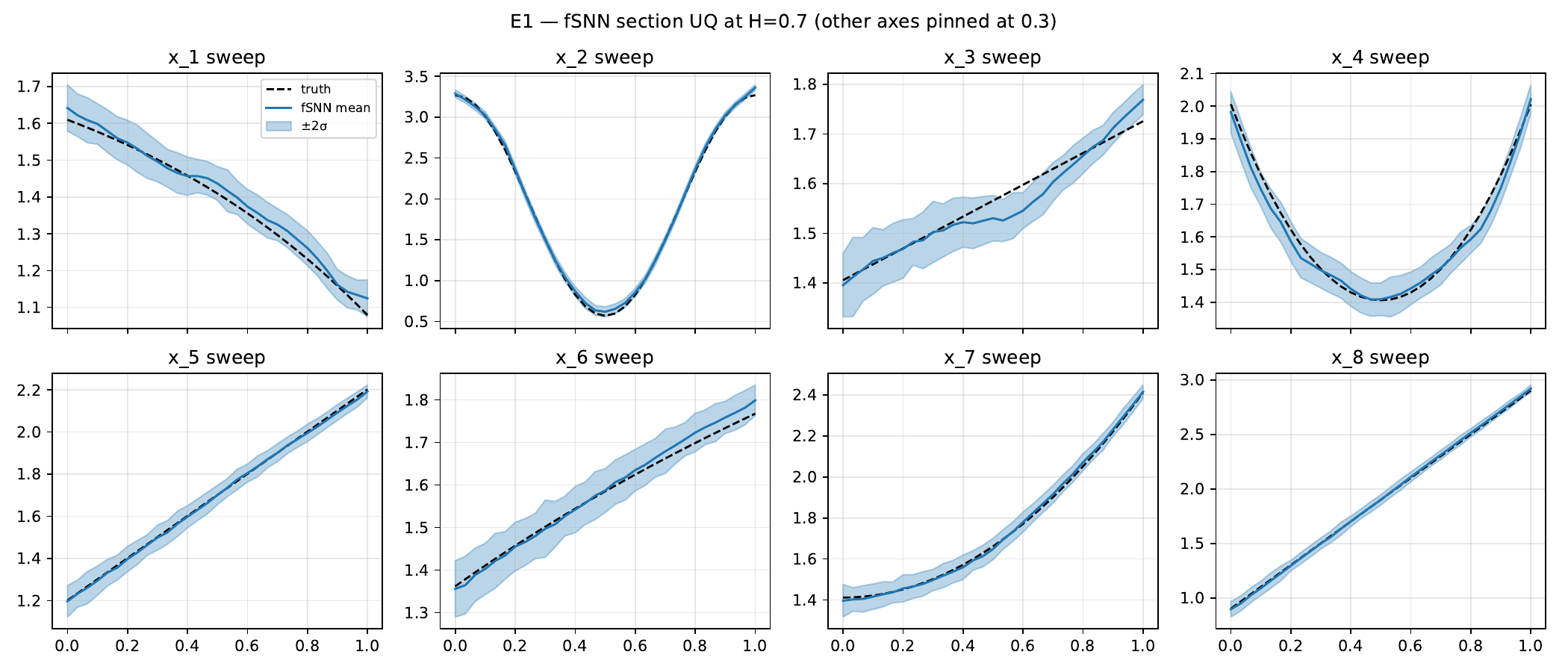}
  \caption{E1(c): predictive mean and $\pm2\hat\sigma$ bands on the
  eight one dimensional sections.  From top to bottom,
  $H=0.3$, $0.5$ and $0.7$.  Dashed curves show the noiseless target
  sections.}
  \label{fig:exp-e1-uq}
\end{figure*}

\subsection{E2: Long Memory Time Series Generation}\label{sec:exp-e2}

We use three synthetic fractional Ornstein--Uhlenbeck series with
reference Hurst exponents $0.7$, $0.8$ and $0.9$, together with five real
series: SPX log close, Nile annual minima, northern hemisphere
temperature, NBSdiff1kg and ethernet traffic.  The compared models are
a deterministic RNN, a Brownian SDE-Net with $H=1/2$, and an FSNN
trained at the estimated target Hurst exponent $\hat H$.  Each model is
run with three seeds.

The primary metrics are chosen to separate scaling, correlation and
one time distributional matching.  The Hurst error
$|\hat H_{\mathrm{gen}}-\hat H_{\mathrm{tgt}}|$ measures recovery of
the scaling exponent.  The weighted ACF score emphasises longer lags,
whereas the ordinary ACF score averages autocorrelation errors without
the same long lag weighting.  The marginal total variation distance
compares the generated and target one time distributions.

\begin{table*}[pos=tb]
  \caption{E2: long memory generation metrics.  Lower is better in all
  columns.  Hurst error measures scaling recovery, wACF emphasises
  long lag dependence, marginal TV measures one time distributional
  matching, and ACF is the unweighted autocorrelation score.}
  \label{tab:exp-e2}
  \centering
  {\footnotesize
  \begin{tabular*}{\textwidth}{@{\extracolsep{\fill}}llcccc@{}}
\toprule
Data Set & Model & $|\hat H_{\mathrm{gen}}-\hat H_{\mathrm{tgt}}|\downarrow$ & wACF $\downarrow$ & Marginal TV $\downarrow$ & ACF $\downarrow$ \\
\midrule
fOU $H{=}0.7$ & RNN & 0.137\,$\pm$\,0.002 & \textbf{0.355\,$\pm$\,0.021} & \textbf{0.106\,$\pm$\,0.009} & \textbf{0.311\,$\pm$\,0.027} \\
 & SDE$_{H=1/2}$ & 0.138\,$\pm$\,0.005 & 0.361\,$\pm$\,0.028 & 0.125\,$\pm$\,0.026 & 0.315\,$\pm$\,0.033 \\
 & \textbf{FSNN}$_{\hat H}$ & \textbf{0.009\,$\pm$\,0.006} & 0.398\,$\pm$\,0.128 & 0.107\,$\pm$\,0.002 & 0.329\,$\pm$\,0.097 \\
\midrule
fOU $H{=}0.8$ & RNN & 0.216\,$\pm$\,0.004 & 0.373\,$\pm$\,0.036 & \textbf{0.110\,$\pm$\,0.019} & 0.424\,$\pm$\,0.038 \\
 & SDE$_{H=1/2}$ & 0.199\,$\pm$\,0.011 & \textbf{0.363\,$\pm$\,0.045} & 0.138\,$\pm$\,0.019 & \textbf{0.421\,$\pm$\,0.046} \\
 & \textbf{FSNN}$_{\hat H}$ & \textbf{0.007\,$\pm$\,0.007} & 0.517\,$\pm$\,0.179 & 0.158\,$\pm$\,0.002 & 0.471\,$\pm$\,0.140 \\
\midrule
fOU $H{=}0.9$ & RNN & 0.227\,$\pm$\,0.053 & 0.718\,$\pm$\,0.230 & \textbf{0.141\,$\pm$\,0.020} & 0.933\,$\pm$\,0.187 \\
 & SDE$_{H=1/2}$ & 0.151\,$\pm$\,0.021 & \textbf{0.591\,$\pm$\,0.261} & 0.151\,$\pm$\,0.008 & \textbf{0.825\,$\pm$\,0.224} \\
 & \textbf{FSNN}$_{\hat H}$ & \textbf{0.024\,$\pm$\,0.011} & 0.982\,$\pm$\,0.143 & 0.234\,$\pm$\,0.132 & 0.950\,$\pm$\,0.086 \\
\midrule
SPX log close & RNN & \textbf{0.013\,$\pm$\,0.003} & 0.783\,$\pm$\,0.042 & \textbf{0.160\,$\pm$\,0.019} & 1.642\,$\pm$\,0.035 \\
 & SDE$_{H=1/2}$ & 0.019\,$\pm$\,0.001 & 0.673\,$\pm$\,0.147 & 0.208\,$\pm$\,0.079 & \textbf{1.132\,$\pm$\,0.107} \\
 & \textbf{FSNN}$_{\hat H}$ & 0.037\,$\pm$\,0.003 & \textbf{0.624\,$\pm$\,0.017} & 0.246\,$\pm$\,0.054 & 1.374\,$\pm$\,0.050 \\
\midrule
NileMin & RNN & 0.104\,$\pm$\,0.003 & \textbf{0.505\,$\pm$\,0.033} & 0.758\,$\pm$\,0.001 & \textbf{0.611\,$\pm$\,0.004} \\
 & SDE$_{H=1/2}$ & \textbf{0.093\,$\pm$\,0.001} & 0.516\,$\pm$\,0.019 & \textbf{0.153\,$\pm$\,0.012} & 0.618\,$\pm$\,0.012 \\
 & \textbf{FSNN}$_{\hat H}$ & 0.103\,$\pm$\,0.001 & 0.716\,$\pm$\,0.142 & 0.353\,$\pm$\,0.065 & 0.774\,$\pm$\,0.269 \\
\midrule
NhemiTemp & RNN & 0.166\,$\pm$\,0.003 & 1.259\,$\pm$\,0.026 & 0.837\,$\pm$\,0.000 & 1.144\,$\pm$\,0.012 \\
 & SDE$_{H=1/2}$ & \textbf{0.160\,$\pm$\,0.004} & 1.245\,$\pm$\,0.037 & \textbf{0.111\,$\pm$\,0.012} & 1.133\,$\pm$\,0.021 \\
 & \textbf{FSNN}$_{\hat H}$ & 0.162\,$\pm$\,0.003 & \textbf{1.160\,$\pm$\,0.218} & 0.351\,$\pm$\,0.047 & \textbf{0.970\,$\pm$\,0.127} \\
\midrule
NBSdiff1kg & RNN & 0.284\,$\pm$\,0.027 & 1.003\,$\pm$\,0.331 & 0.829\,$\pm$\,0.008 & 0.940\,$\pm$\,0.275 \\
 & SDE$_{H=1/2}$ & \textbf{0.250\,$\pm$\,0.007} & 0.770\,$\pm$\,0.089 & \textbf{0.177\,$\pm$\,0.022} & 0.788\,$\pm$\,0.060 \\
 & \textbf{FSNN}$_{\hat H}$ & 0.253\,$\pm$\,0.017 & \textbf{0.726\,$\pm$\,0.137} & 0.469\,$\pm$\,0.402 & \textbf{0.663\,$\pm$\,0.069} \\
\midrule
ethernet & RNN & 0.274\,$\pm$\,0.006 & 1.616\,$\pm$\,0.042 & 0.484\,$\pm$\,0.008 & 1.521\,$\pm$\,0.045 \\
 & SDE$_{H=1/2}$ & 0.266\,$\pm$\,0.003 & 1.597\,$\pm$\,0.044 & 0.529\,$\pm$\,0.009 & 1.505\,$\pm$\,0.041 \\
 & \textbf{FSNN}$_{\hat H}$ & \textbf{0.259\,$\pm$\,0.019} & \textbf{1.419\,$\pm$\,0.359} & \textbf{0.427\,$\pm$\,0.175} & \textbf{1.219\,$\pm$\,0.476} \\

\end{tabular*}\par}
\end{table*}

\Cref{tab:exp-e2} separates the scaling and distributional effects.  On the three synthetic fOU
series, FSNN$_{\hat H}$ recovers the target Hurst exponent much more
accurately than both baselines: the Hurst errors are $0.009$, $0.007$
and $0.024$, while the Brownian and RNN errors are one order of
magnitude larger.  This is the clearest evidence that matching the
driver Hurst exponent helps recover long memory scaling.  The ACF and
marginal columns show a more nuanced picture.  On the fOU series,
Brownian SDE-Net or RNN can be better on marginal TV and ACF even when
they miss the scaling exponent.  On real data, FSNN$_{\hat H}$ is
competitive or best in long lag dependence on SPX, NhemiTemp,
NBSdiff1kg and ethernet traffic, but Brownian SDE-Net often has the
best marginal TV\@.  We therefore interpret E2 as evidence for a
long memory inductive bias, not as uniform superiority on every
distributional statistic.

\FloatBarrier

\subsection{E3: Image Classification and Structured Perturbations}
\label{sec:exp-e3}

E3 uses a small convolutional backbone on MNIST, Fashion-MNIST and
CIFAR-10.  We compare a deterministic CNN, a Brownian SNN ($H=0.5$),
and FSNNs on the available grid $H\in\{0.1,0.2,\ldots,0.9\}$.  All
models are trained on clean images, and corruptions are introduced only at
test time.  MNIST and Fashion-MNIST are trained for 20 epochs, while
CIFAR-10 is trained for 50 epochs.  The clean, common corruption and
fBm field results below use seed $2026$.

\subsubsection{Clean Accuracy and Common Corruptions}
\Cref{tab:exp-e3-common} gives the
clean accuracy and the average accuracy under five common corruptions.
Each corruption average is taken over five severity levels, and mCA is
the mean of the five corruption columns.  For the three image data sets,
the severity curves in
\cref{fig:exp-e3-sev}
show the same data before averaging.

\begin{table*}[pos=t]
  \caption{E3: clean accuracy and common corruption accuracy on
  MNIST, Fashion-MNIST and CIFAR-10.  Corruption entries are averaged
  over five severities, and CIFAR-10 models are trained for 50 epochs.}
  \label{tab:exp-e3-common}
  \centering
  {\footnotesize
    \begin{tabular*}{\textwidth}{@{\extracolsep{\fill}}lccccccc@{}}
      \toprule
      \multicolumn{8}{c}{\textbf{(a) MNIST}} \\
      \midrule
      Model & Clean & mCA & Gauss. & Motion & Defocus & Fog & Frost \\
    \midrule
    CNN & \textbf{0.9900} & \textbf{0.7191} & \textbf{0.8599} & \textbf{0.8858} & \textbf{0.7952} & \textbf{0.5697} & 0.4850 \\
    FSNN $H{=}0.1$ & 0.9894 & 0.6862 & 0.7940 & 0.8020 & 0.7031 & 0.4595 & 0.6726 \\
    FSNN $H{=}0.2$ & 0.9891 & 0.6886 & 0.8033 & 0.8128 & 0.7140 & 0.4652 & 0.6479 \\
    FSNN $H{=}0.3$ & 0.9872 & 0.6986 & 0.8111 & 0.8180 & 0.7260 & 0.4762 & 0.6616 \\
    FSNN $H{=}0.4$ & 0.9890 & 0.7045 & 0.8139 & 0.8252 & 0.7299 & 0.4821 & 0.6716 \\
    SNN $H{=}0.5$ & 0.9878 & 0.7008 & 0.8053 & 0.8194 & 0.7132 & 0.4790 & 0.6872 \\
    FSNN $H{=}0.6$ & 0.9884 & 0.6966 & 0.8121 & 0.8203 & 0.7215 & 0.4772 & 0.6517 \\
    FSNN $H{=}0.7$ & 0.9878 & 0.6958 & 0.7893 & 0.8011 & 0.6949 & 0.4996 & \textbf{0.6942} \\
    FSNN $H{=}0.8$ & 0.9885 & 0.6767 & 0.7723 & 0.7810 & 0.6838 & 0.5039 & 0.6426 \\
    FSNN $H{=}0.9$ & 0.9884 & 0.6788 & 0.7577 & 0.7611 & 0.6689 & 0.5205 & 0.6859 \\
    \bottomrule
    \end{tabular*}
    \par\medskip
    \begin{tabular*}{\textwidth}{@{\extracolsep{\fill}}lccccccc@{}}
      \toprule
      \multicolumn{8}{c}{\textbf{(b) Fashion-MNIST}} \\
      \midrule
      Model & Clean & mCA & Gauss. & Motion & Defocus & Fog & Frost \\
    \midrule
    CNN & 0.9024 & 0.5689 & 0.6761 & 0.7054 & 0.6420 & 0.3337 & 0.4873 \\
    FSNN $H{=}0.1$ & 0.9041 & 0.6500 & 0.7100 & \textbf{0.7418} & 0.6480 & 0.5180 & 0.6322 \\
    FSNN $H{=}0.2$ & 0.9039 & \textbf{0.6538} & \textbf{0.7135} & 0.7330 & \textbf{0.6576} & \textbf{0.5365} & 0.6284 \\
    FSNN $H{=}0.3$ & \textbf{0.9047} & 0.6484 & 0.7056 & 0.7394 & 0.6435 & 0.5111 & 0.6422 \\
    FSNN $H{=}0.4$ & 0.9043 & 0.6357 & 0.6920 & 0.7164 & 0.6443 & 0.4850 & 0.6411 \\
    SNN $H{=}0.5$ & 0.9016 & 0.6459 & 0.6892 & 0.7183 & 0.6387 & 0.5239 & 0.6593 \\
    FSNN $H{=}0.6$ & 0.9001 & 0.6484 & 0.6971 & 0.7311 & 0.6469 & 0.5050 & 0.6619 \\
    FSNN $H{=}0.7$ & 0.9008 & 0.6464 & 0.6975 & 0.7338 & 0.6488 & 0.4634 & \textbf{0.6883} \\
    FSNN $H{=}0.8$ & 0.9018 & 0.6436 & 0.6898 & 0.7175 & 0.6465 & 0.4853 & 0.6789 \\
    FSNN $H{=}0.9$ & 0.9000 & 0.6458 & 0.6837 & 0.7130 & 0.6403 & 0.5080 & 0.6842 \\
    \bottomrule
    \end{tabular*}
    \par\medskip
    \begin{tabular*}{\textwidth}{@{\extracolsep{\fill}}lccccccc@{}}
      \toprule
      \multicolumn{8}{c}{\textbf{(c) CIFAR-10}} \\
      \midrule
      Model & Clean & mCA & Gauss. & Motion & Defocus & Fog & Frost \\
    \midrule
    CNN & \textbf{0.6818} & 0.3825 & 0.3701 & 0.3776 & 0.3458 & \textbf{0.4406} & 0.3783 \\
    FSNN $H{=}0.1$ & 0.6699 & 0.3909 & 0.3866 & 0.3956 & 0.3613 & 0.4059 & 0.4052 \\
    FSNN $H{=}0.2$ & 0.6670 & 0.3984 & 0.3963 & 0.4076 & 0.3714 & 0.3903 & 0.4261 \\
    FSNN $H{=}0.3$ & 0.6763 & 0.3952 & 0.3953 & 0.4066 & 0.3722 & 0.3953 & 0.4067 \\
    FSNN $H{=}0.4$ & 0.6774 & 0.3992 & 0.3950 & 0.4010 & 0.3725 & 0.4114 & 0.4162 \\
    SNN $H{=}0.5$ & 0.6746 & 0.4054 & 0.4021 & 0.4090 & 0.3803 & 0.4020 & 0.4337 \\
    FSNN $H{=}0.6$ & 0.6758 & \textbf{0.4149} & 0.4082 & 0.4168 & 0.3835 & 0.4193 & \textbf{0.4469} \\
    FSNN $H{=}0.7$ & 0.6747 & 0.4046 & 0.3873 & 0.3952 & 0.3621 & 0.4321 & 0.4462 \\
    FSNN $H{=}0.8$ & 0.6725 & 0.3884 & 0.3820 & 0.3873 & 0.3611 & 0.4090 & 0.4029 \\
    FSNN $H{=}0.9$ & 0.6693 & 0.4119 & \textbf{0.4126} & \textbf{0.4182} & \textbf{0.3881} & 0.4097 & 0.4312 \\
    \bottomrule
  \end{tabular*}\par}
\end{table*}

\begin{figure*}[pos=t]
  \centering
  \includegraphics[width=\linewidth]{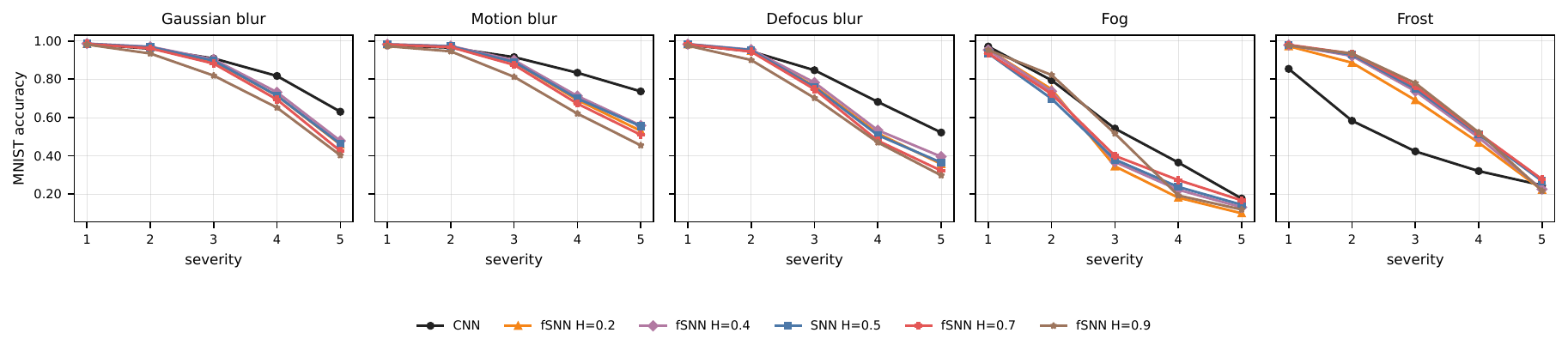}
  \par\smallskip
  \includegraphics[width=\linewidth]{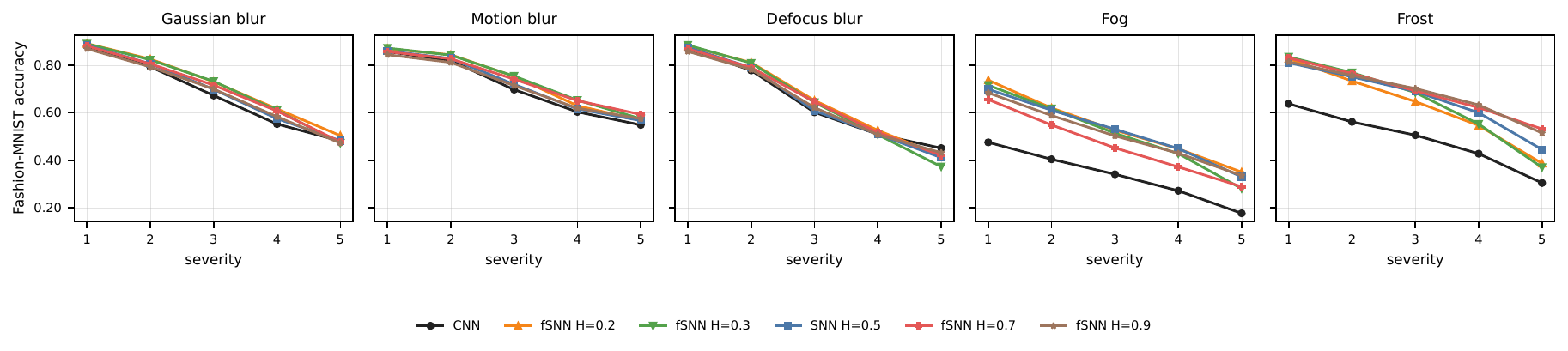}
  \par\smallskip
  \includegraphics[width=\linewidth]{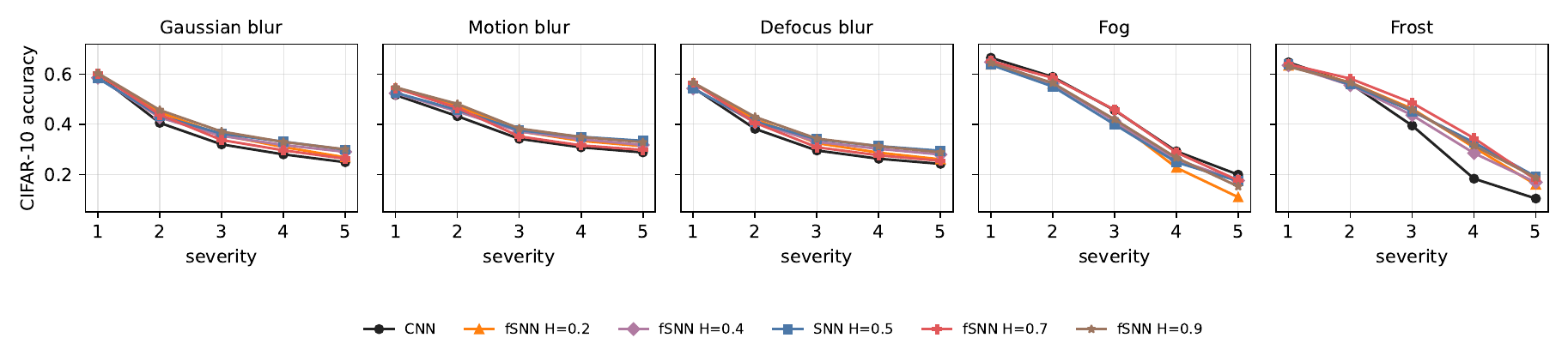}
  \caption{E3: accuracy as corruption severity increases from 1 to 5.
  From top to bottom, MNIST, Fashion-MNIST and CIFAR-10.  The plotted
  FSNN rows represent rough, Brownian and smooth driver regimes, and the
  full Hurst grid is summarised in \cref{tab:exp-e3-common}.}
  \label{fig:exp-e3-sev}
\end{figure*}

On MNIST, the deterministic CNN is best on clean
accuracy and on Gaussian, motion, defocus and fog corruptions, while the
FSNN with $H=0.7$ is best under frost.  This is consistent with the
fact that MNIST is almost saturated by the deterministic backbone:
clean accuracy is already about $99\%$, and the simple digit
manifold leaves little room for stochastic smoothing to improve the
average low frequency corruption score.  The stochastic models mainly
close the gap for frost and high severity fog, where the perturbation
is less aligned with the invariances learned by the CNN\@.

Fashion-MNIST has a different pattern.  The clean accuracies of all
stochastic models remain within a narrow range around the CNN, but
their corruption accuracy is substantially higher.  The best mCA is
$0.6538$ at $H=0.2$, compared with $0.5689$ for the CNN\@.  The gain is
not concentrated in a single corruption: rough FSNNs improve Gaussian
blur, defocus blur and fog, $H=0.1$ gives the best motion blur average,
and $H=0.7$ gives the best frost average.  CIFAR-10 gives the clearest
medium complexity image result.  The CNN has the highest clean
accuracy, but the best FSNN, obtained at $H=0.6$, raises mCA from
$0.3825$ to $0.4149$ while losing less than one percentage point of
clean accuracy.  The gain is visible on blur and frost corruptions,
whereas the CNN remains strongest under fog.  Thus the useful Hurst
regime depends on both the data set and the perturbation spectrum.

\subsubsection{Structured fBm Field Perturbations}
The second robustness probe adds normalised two dimensional fBm random
fields to the input image.  The field amplitude is
$\varepsilon\in\{0,0.1,0.2\}$, and the perturbation field is either
$H_{\mathrm{atk}}=0.5$ or $H_{\mathrm{atk}}=0.9$.  This is not a
worst case adversarial attack, but a controlled test of structured
additive noise.  \Cref{fig:exp-e3-fbm-strength} displays the full
amplitude sweep for representative training Hurst values, and
\cref{tab:exp-e3-fbm} reports the strongest perturbation
($\varepsilon=0.2$) on the full Hurst grid.

\begin{figure*}[pos=t]
  \centering
  \includegraphics[width=0.70\linewidth]{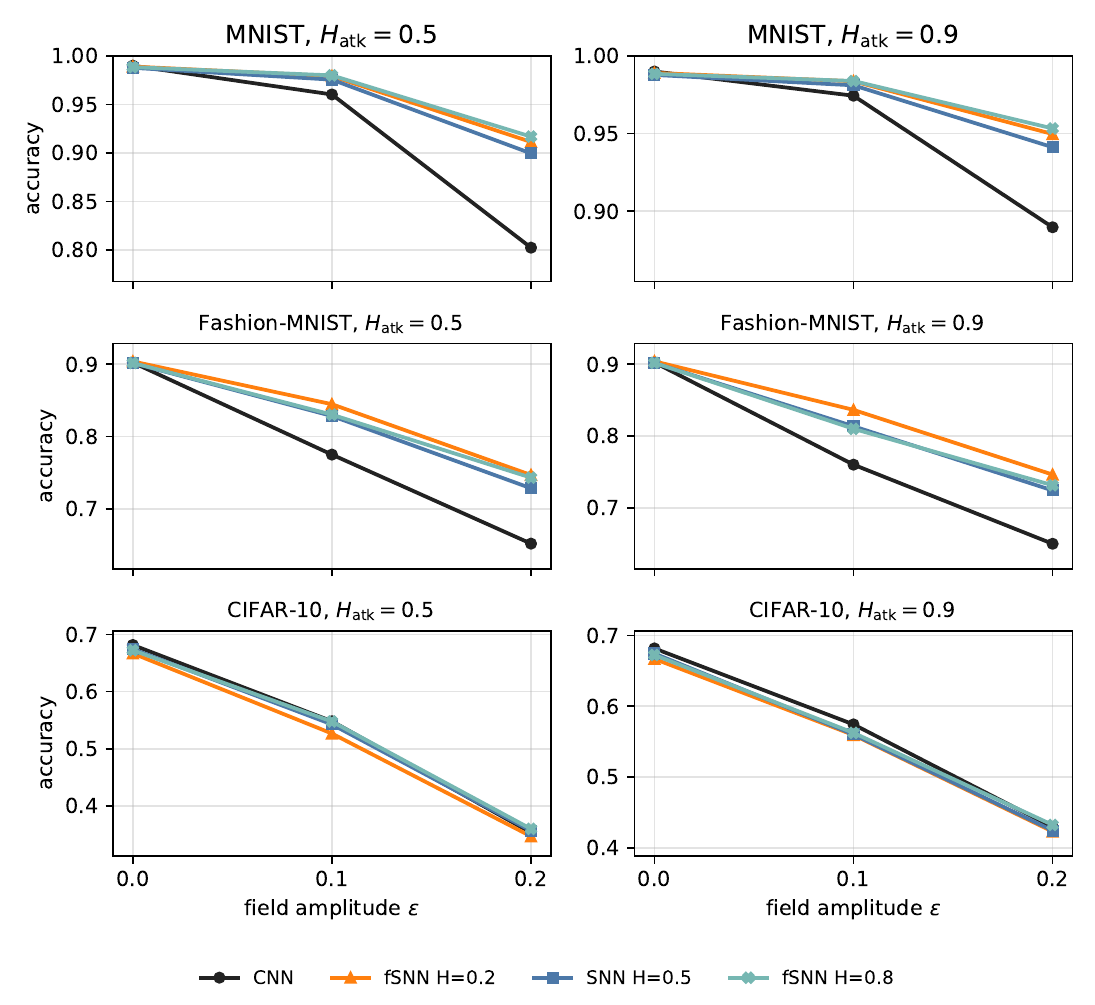}
  \caption{E3: accuracy under additive fBm random field perturbations
  as the field amplitude $\varepsilon$ increases.  Rows correspond to
  MNIST, Fashion-MNIST, and CIFAR-10, while the two columns use
  Brownian spectrum and smooth spectrum perturbation fields.}
  \label{fig:exp-e3-fbm-strength}
\end{figure*}

\begin{table*}[pos=tb]
  \caption{E3: accuracy under normalised fBm field perturbations with
  amplitude $\varepsilon=0.2$.  Higher is better.}
  \label{tab:exp-e3-fbm}
  \centering
  {\footnotesize
  \begin{tabular*}{\textwidth}{@{\extracolsep{\fill}}lcccccc@{}}
    \toprule
     & \multicolumn{2}{c}{MNIST} & \multicolumn{2}{c}{Fashion-MNIST}
     & \multicolumn{2}{c}{CIFAR-10} \\
    \cmidrule(lr){2-3}\cmidrule(lr){4-5}\cmidrule(lr){6-7}
    Model & $H_{\mathrm{atk}}{=}0.5$ & $H_{\mathrm{atk}}{=}0.9$ &
    $H_{\mathrm{atk}}{=}0.5$ & $H_{\mathrm{atk}}{=}0.9$ &
    $H_{\mathrm{atk}}{=}0.5$ & $H_{\mathrm{atk}}{=}0.9$ \\
    \midrule
    CNN & 0.8022 & 0.8896 & 0.6518 & 0.6502 & 0.3528 & 0.4294 \\
    FSNN $H{=}0.1$ & 0.9154 & 0.9497 & \textbf{0.7509} & \textbf{0.7481} & 0.3601 & 0.4298 \\
    FSNN $H{=}0.2$ & 0.9113 & 0.9497 & 0.7468 & 0.7464 & 0.3463 & 0.4230 \\
    FSNN $H{=}0.3$ & 0.8956 & 0.9409 & 0.7425 & 0.7369 & 0.3428 & 0.4282 \\
    FSNN $H{=}0.4$ & 0.9041 & 0.9498 & 0.7279 & 0.7157 & 0.3513 & 0.4231 \\
    SNN $H{=}0.5$ & 0.8995 & 0.9412 & 0.7284 & 0.7244 & 0.3560 & 0.4244 \\
    FSNN $H{=}0.6$ & 0.8999 & 0.9397 & 0.7228 & 0.7233 & 0.3481 & 0.4256 \\
    FSNN $H{=}0.7$ & 0.9076 & 0.9438 & 0.7205 & 0.7168 & 0.3498 & 0.4202 \\
    FSNN $H{=}0.8$ & \textbf{0.9168} & \textbf{0.9533} & 0.7431 & 0.7316 & 0.3593 & \textbf{0.4324} \\
    FSNN $H{=}0.9$ & 0.9136 & 0.9515 & 0.7222 & 0.7081 & \textbf{0.3627} & 0.4141 \\
    \bottomrule
  \end{tabular*}\par}
\end{table*}
The structured field result is more favourable to stochastic
backbones than the common corruption result on MNIST\@.  At
$\varepsilon=0.2$, every FSNN/SNN row is higher than the CNN for both
attacker spectra.  The best MNIST row is $H=0.8$, with accuracies
$0.9168$ and $0.9533$ under $H_{\mathrm{atk}}=0.5$ and $0.9$,
respectively.  Fashion-MNIST shows the same separation from the CNN,
but the best row shifts to the rough regime $H=0.1$.  The curves also
show that the ranking is already visible at $\varepsilon=0.1$ and
widens as the amplitude increases.  On CIFAR-10 the same perturbation
does not produce a uniform separation: several FSNNs are slightly
better than the CNN, but the differences are small compared with the
common corruption gains. 

The E3 conclusion is therefore spectrum dependent.  On MNIST, the CNN
remains strongest for clean and most common corruption accuracy,
probably because the task is already close to saturated for a small
CNN, while the stochastic models are much stronger under additive fBm
fields.  On Fashion-MNIST, stochastic backbones improve both common
corruptions and fBm field perturbations.  On CIFAR-10, FSNNs do not
beat the CNN on clean accuracy, but they give a clear common corruption
improvement, with the best result at $H=0.6$.  No single Hurst exponent
dominates all columns: rough drivers are best for several
Fashion-MNIST perturbations, smooth drivers are best for the MNIST
fBm field test, and an intermediate Hurst exponent is best for CIFAR-10
common corruptions.

\subsection{Summary}\label{sec:exp-summary}

The experiments support different parts of the paper at different
levels of strictness.  E1(a) is a direct control error diagnostic: in
the additive diffusion case covered by Corollary~\ref{cor:scaling}, the
measured error decreases over the tested budget consistently with a
$K^{-1}$ upper envelope interpretation, and the finite budget slopes are
not sharp exponent estimates.  E1(b)--(c) form a higher dimensional
architecture level consistency and UQ test under direct
backpropagation with Adam. The test error reaches the same order as the
label noise reference scale and the sectionwise uncertainty bands are
calibrated.
E2 provides the main long memory evidence: matching the Hurst exponent
recovers the fOU scaling exponent with small error and improves several
long lag correlation metrics on real data, while marginal
distributional fit remains task dependent.  E3 is an application level
robustness study: FSNNs are not uniformly better than CNNs on clean
vision accuracy, and MNIST is too saturated to show a common corruption
advantage.  The benefit becomes clearer on Fashion-MNIST and CIFAR-10,
where stochastic backbones improve mean common corruption accuracy, and the
structured fBm field test also favours stochastic models on MNIST and
Fashion-MNIST.

\section{Conclusions}\label{sec:conclusion}

In this paper, we developed a numerical analysis framework for training
fractional stochastic neural networks driven by normalised fractional
Brownian increments.  We formulated the model on a discrete depth grid,
derived the associated backward stochastic difference equation and
samplewise backpropagation formula from the discrete stochastic
maximum principle.  We proved a mean square convergence bound for
projected samplewise SGD on deterministic network parameters and
reported numerical experiments on convergence, uncertainty
quantification, long memory modelling and robustness.

\FloatBarrier

\section*{Funding}
This work was supported by the National Natural Science Foundation of
China (Grant number 12471417) and the National Key R\&D Program of China
(Grant number 2023YFA1009200). The funders had no role in the study design,
data collection and analysis, preparation of the manuscript, or decision
to submit the manuscript for publication.

\section*{Declaration of competing interest}
The authors declare that they have no known competing financial interests
or personal relationships that could have appeared to influence the work
reported in this paper.

\section*{Data availability}
The data and code supporting the findings of this study are available from
the corresponding author upon reasonable request.

\printcredits

\section*{Declaration of generative AI and AI-assisted technologies in the manuscript preparation process}

During the preparation of this work, the authors used ChatGPT for grammar checking and language polishing. The author(s) reviewed and edited the output as needed and take full responsibility for the content of the published article.

\printbibliography

@book{frizvictoir2010,
  doi       = {10.1017/CBO9780511845079},
  title     = {Multidimensional Stochastic Processes as Rough Paths: Theory and Applications},
  author    = {Friz, Peter K. and Victoir, Nicolas B.},
  publisher = {Cambridge University Press},
  series    = {Cambridge Studies in Advanced Mathematics},
  number    = {120},
  year      = {2010}
}

@inproceedings{morrill2021neural,
  url       = {https://proceedings.mlr.press/v139/morrill21b.html},
  publisher = {PMLR},
  series    = {Proceedings of Machine Learning Research},
  volume    = {139},
  title     = {Neural Rough Differential Equations for Long Time Series},
  author    = {Morrill, James and Salvi, Cristopher and Kidger, Patrick and Foster, James},
  booktitle = {Proceedings of the 38th International Conference on Machine Learning (ICML)},
  pages     = {7829--7838},
  year      = {2021}
}

@article{gubinelli2004controlling,
  doi     = {10.1016/j.jfa.2004.01.002},
  title   = {Controlling rough paths},
  author  = {Gubinelli, Massimiliano},
  journal = {Journal of Functional Analysis},
  volume  = {216},
  number  = {1},
  pages   = {86--140},
  year    = {2004}
}

@article{hanli2025discrete,
  doi     = {10.1016/j.sysconle.2025.106202},
  title   = {Maximum principle for discrete-time control systems driven by fractional noises and related backward stochastic difference equations},
  author  = {Han, Yuecai and Li, Yuhang},
  journal = {Systems \& Control Letters},
  volume  = {204},
  eid     = {106202},
  year    = {2025}
}

@article{archibald2024,
  doi     = {10.1137/22M1523765},
  title   = {Numerical Analysis for Convergence of a Sample-wise Backpropagation Method for Training Stochastic Neural Networks},
  author  = {Archibald, Richard and Bao, Feng and Cao, Yanzhao and Sun, Hui},
  journal = {SIAM Journal on Numerical Analysis},
  volume  = {62},
  number  = {2},
  pages   = {593--621},
  year    = {2024}
}

@article{archibald2022,
  doi     = {10.3934/dcdss.2022062},
  title   = {A Backward {SDE} Method for Uncertainty Quantification in Deep Learning},
  author  = {Archibald, Richard and Bao, Feng and Cao, Yanzhao and Zhang, He},
  journal = {Discrete and Continuous Dynamical Systems - Series S},
  volume  = {15},
  number  = {10},
  pages   = {2807--2835},
  year    = {2022}
}

@article{hanhusong2013,
  doi     = {10.1007/s00245-012-9188-7},
  title   = {Maximum Principle for General Controlled Systems Driven by Fractional {B}rownian Motions},
  author  = {Han, Yuecai and Hu, Yaozhong and Song, Jian},
  journal = {Applied Mathematics \& Optimization},
  volume  = {67},
  number  = {2},
  pages   = {279--322},
  year    = {2013}
}

@article{hu2009backward,
  doi     = {10.1137/070709451},
  title   = {Backward Stochastic Differential Equation Driven by Fractional {B}rownian Motion},
  author  = {Hu, Yaozhong and Peng, Shige},
  journal = {SIAM Journal on Control and Optimization},
  volume  = {48},
  number  = {3},
  pages   = {1675--1700},
  year    = {2009}
}

@inproceedings{hayashi2022,
  doi       = {10.1109/DSAA54385.2022.10032351},
  publisher = {IEEE},
  title     = {Fractional {SDE}-Net: Generation of Time Series Data with Long-term Memory},
  author    = {Hayashi, Kohei and Nakagawa, Kei},
  booktitle = {2022 IEEE 9th International Conference on Data Science and Advanced Analytics (DSAA)},
  pages     = {1--10},
  year      = {2022}
}

@inproceedings{kong2020,
  url       = {https://proceedings.mlr.press/v119/kong20b.html},
  publisher = {PMLR},
  title     = {{SDE}-Net: Equipping Deep Neural Networks with Uncertainty Estimates},
  author    = {Kong, Lingkai and Sun, Jimeng and Zhang, Chao},
  booktitle = {Proceedings of the 37th International Conference on Machine Learning},
  series    = {Proceedings of Machine Learning Research},
  volume    = {119},
  pages     = {5405--5415},
  year      = {2020}
}

@article{detommaso2024fortuna,
  url     = {https://www.jmlr.org/papers/v25/23-0145.html},
  author  = {Gianluca Detommaso and Alberto Gasparin and Michele Donini and Matthias Seeger and Andrew Gordon Wilson and Cedric Archambeau},
  title   = {{Fortuna}: A Library for Uncertainty Quantification in Deep Learning},
  journal = {Journal of Machine Learning Research},
  year    = {2024},
  volume  = {25},
  number  = {238},
  pages   = {1--7}
}

@inproceedings{chen2018neuralode,
  url       = {https://papers.nips.cc/paper/2018/hash/69386f6bb1dfed68692a24c8686939b9-Abstract.html},
  title     = {Neural Ordinary Differential Equations},
  author    = {Chen, Ricky T. Q. and Rubanova, Yulia and Bettencourt, Jesse and Duvenaud, David},
  booktitle = {Advances in Neural Information Processing Systems},
  volume    = {31},
  pages     = {6571--6583},
  year      = {2018}
}

@article{marzouk2024distribution,
  url     = {https://jmlr.org/papers/v25/23-1280.html},
  author  = {Youssef Marzouk and Zhi Ren and Sven Wang and Jakob Zech},
  title   = {Distribution Learning via Neural Differential Equations: A Nonparametric Statistical Perspective},
  journal = {Journal of Machine Learning Research},
  year    = {2024},
  volume  = {25},
  number  = {232},
  pages   = {1--61}
}

@inproceedings{he2016resnet,
  url       = {https://openaccess.thecvf.com/content_cvpr_2016/html/He_Deep_Residual_Learning_CVPR_2016_paper.html},
  title     = {Deep Residual Learning for Image Recognition},
  author    = {He, Kaiming and Zhang, Xiangyu and Ren, Shaoqing and Sun, Jian},
  booktitle = {IEEE Conference on Computer Vision and Pattern Recognition (CVPR)},
  pages     = {770--778},
  year      = {2016}
}

@inproceedings{kidger2021neuralsde,
  url       = {https://proceedings.mlr.press/v139/kidger21b.html},
  publisher = {PMLR},
  title     = {Neural {SDE}s as Infinite-Dimensional {GAN}s},
  author    = {Kidger, Patrick and Foster, James and Li, Xuechen and Oberhauser, Harald and Lyons, Terry},
  booktitle = {Proceedings of the 38th International Conference on Machine Learning},
  series    = {Proceedings of Machine Learning Research},
  volume    = {139},
  pages     = {5453--5463},
  year      = {2021}
}

@inproceedings{kidger2020neuralcde,
  url       = {https://papers.nips.cc/paper/2020/hash/4a5876b450b45371f6cfe5047ac8cd45-Abstract.html},
  title     = {Neural Controlled Differential Equations for Irregular Time Series},
  author    = {Kidger, Patrick and Morrill, James and Foster, James and Lyons, Terry},
  booktitle = {Advances in Neural Information Processing Systems (NeurIPS)},
  volume    = {33},
  pages     = {6696--6707},
  year      = {2020}
}

@inproceedings{li2020scalablesde,
  url       = {https://proceedings.mlr.press/v108/li20i.html},
  publisher = {PMLR},
  title     = {Scalable Gradients for Stochastic Differential Equations},
  author    = {Li, Xuechen and Wong, Ting-Kam Leonard and Chen, Ricky T. Q. and Duvenaud, David},
  booktitle = {Proceedings of the 23rd International Conference on Artificial Intelligence and Statistics},
  series    = {Proceedings of Machine Learning Research},
  volume    = {108},
  pages     = {3870--3882},
  year      = {2020}
}

@online{liu2019neuralsde,
  doi       = {10.48550/arXiv.1906.02355},
  title     = {Neural {SDE}: Stabilizing Neural {ODE} Networks with Stochastic Noise},
  author    = {Liu, Xuanqing and Xiao, Tesi and Si, Si and Cao, Qin and Kumar, Sanjiv and Hsieh, Cho-Jui},
  entrysubtype = {Preprint},
  organization = {arXiv},
  year      = {2019}
}

@online{tzen2019neural,
  doi       = {10.48550/arXiv.1905.09883},
  title     = {Neural Stochastic Differential Equations: Deep Latent {G}aussian Models in the Diffusion Limit},
  author    = {Tzen, Belinda and Raginsky, Maxim},
  entrysubtype = {Preprint},
  organization = {arXiv},
  year      = {2019}
}

@article{mandelbrot1968fbm,
  doi     = {10.1137/1010093},
  title   = {Fractional {B}rownian Motions, Fractional Noises and Applications},
  author  = {Mandelbrot, Benoit B. and Van~Ness, John W.},
  journal = {SIAM Review},
  volume  = {10},
  number  = {4},
  pages   = {422--437},
  year    = {1968}
}

@book{biagini2008fbm,
  doi       = {10.1007/978-1-84628-797-8},
  title     = {Stochastic Calculus for Fractional {B}rownian Motion and Applications},
  author    = {Biagini, Francesca and Hu, Yaozhong and {\O}ksendal, Bernt and Zhang, Tusheng},
  publisher = {Springer},
  year      = {2008}
}

@article{dieker2003simulation,
  doi={10.1017/S0269964803173081},
  title={On spectral simulation of fractional {B}rownian motion},
  author={Dieker, Antonius Bernardus and Mandjes, Michael},
  journal={Probability in the Engineering and Informational Sciences},
  volume={17},
  number={3},
  pages={417--434},
  year={2003},
  publisher={Cambridge University Press}
}

@article{daviesharte1987tests,
  doi     = {10.1093/biomet/74.1.95},
  title   = {Tests for {H}urst Effect},
  author  = {Davies, Robert B. and Harte, D. S.},
  journal = {Biometrika},
  volume  = {74},
  number  = {1},
  pages   = {95--101},
  year    = {1987}
}

@article{woodchan1994simulation,
  doi     = {10.1080/10618600.1994.10474655},
  title   = {Simulation of Stationary {G}aussian Processes in {$[0,1]^d$}},
  author  = {Wood, Andrew T. A. and Chan, Grace},
  journal = {Journal of Computational and Graphical Statistics},
  volume  = {3},
  number  = {4},
  pages   = {409--432},
  year    = {1994}
}

@article{pardoux1990adapted,
  doi     = {10.1016/0167-6911(90)90082-6},
  title   = {Adapted Solution of a Backward Stochastic Differential Equation},
  author  = {Pardoux, Etienne and Peng, S. G.},
  journal = {Systems \& Control Letters},
  volume  = {14},
  number  = {1},
  pages   = {55--61},
  year    = {1990}
}

@book{yongzhou1999stochastic,
  doi       = {10.1007/978-1-4612-1466-3},
  title     = {Stochastic Controls: {H}amiltonian Systems and {HJB} Equations},
  author    = {Yong, Jiongmin and Zhou, Xun Yu},
  publisher = {Springer},
  year      = {1999}
}

@article{granger1980introduction,
  doi     = {10.1111/j.1467-9892.1980.tb00297.x},
  title   = {An Introduction to Long-Memory Time Series Models and Fractional Differencing},
  author  = {Granger, C. W. J. and Joyeux, Roselyne},
  journal = {Journal of Time Series Analysis},
  volume  = {1},
  number  = {1},
  pages   = {15--29},
  year    = {1980}
}

@article{hosking1981fractional,
  doi     = {10.1093/biomet/68.1.165},
  title   = {Fractional Differencing},
  author  = {Hosking, J. R. M.},
  journal = {Biometrika},
  volume  = {68},
  number  = {1},
  pages   = {165--176},
  year    = {1981}
}

@book{beran1994longrange,
  url       = {https://www.routledge.com/Statistics-for-Long-Memory-Processes/Beran/p/book/9780412049019},
  title     = {Statistics for Long-Memory Processes},
  author    = {Beran, Jan},
  publisher = {Chapman \& Hall},
  year      = {1994},
  series    = {Monographs on Statistics and Applied Probability},
  number    = {61}
}

@article{comte1998fractional,
  doi     = {10.1111/1467-9965.00057},
  title   = {Long Memory in Continuous-Time Stochastic Volatility Models},
  author  = {Comte, Fabienne and Renault, Eric},
  journal = {Mathematical Finance},
  volume  = {8},
  number  = {4},
  pages   = {291--323},
  year    = {1998}
}

@article{gatheral2018roughvol,
  doi     = {10.1080/14697688.2017.1393551},
  title   = {Volatility is Rough},
  author  = {Gatheral, Jim and Jaisson, Thibault and Rosenbaum, Mathieu},
  journal = {Quantitative Finance},
  volume  = {18},
  number  = {6},
  pages   = {933--949},
  year    = {2018}
}

@article{elkaroui1997backward,
  doi     = {10.1111/1467-9965.00022},
  title   = {Backward Stochastic Differential Equations in Finance},
  author  = {El Karoui, Nicole and Peng, Shige and Quenez, Marie Claire},
  journal = {Mathematical Finance},
  volume  = {7},
  number  = {1},
  pages   = {1--71},
  year    = {1997}
}

@article{zhang2004numerical,
  doi     = {10.1214/aoap/1075828058},
  title   = {A Numerical Scheme for {BSDE}s},
  author  = {Zhang, Jianfeng},
  journal = {The Annals of Applied Probability},
  volume  = {14},
  number  = {1},
  pages   = {459--488},
  year    = {2004}
}
\end{document}